\documentclass[A4paper, 11pt]{article}

\usepackage{amsmath,amssymb,amsthm}   
\usepackage{makeidx}
\usepackage{mathrsfs}  
\usepackage{mathtools}
\usepackage{bbm}
\usepackage{mathabx}   %
\usepackage{authblk}   

\newtheorem{theorem}{Theorem}[section]
\newtheorem*{theorem*}{Theorem}

\newtheorem{lemma}[theorem]{Lemma}
\newtheorem{definition}[theorem]{Definition}
\newtheorem*{definition*}{Definition}
\newtheorem{proposition}[theorem]{Proposition}
\numberwithin{equation}{section}

\theoremstyle{definition}
\newtheorem{remark}[theorem]{Remark}
\newtheorem{example}[theorem]{Example}

\newcommand\restr[2]{{
  \left.\kern-\nulldelimiterspace 
  #1 
  \vphantom{\big|} 
  \right|_{#2} 
  }}

\renewcommand{\mathbb}{\mathbbm}                     
\renewcommand{\epsilon}{\varepsilon}                 
\renewcommand{\phi}{\varphi}
\renewcommand{\theta}{\vartheta}
\renewcommand{\le}{\leqslant}
\renewcommand{\ge}{\geqslant}

\newcommand{\abs}[1]{\left\lvert #1 \right\rvert}    
\newcommand{\norm}[1]{\left\lVert #1 \right\rVert}   
\newcommand{\scapro}[2]{\langle #1,#2\rangle}

\newcommand{\R}{\mathbb{R}}
\newcommand{\Rp}{\mathbb{R}_{+}}
\newcommand{\N}{\mathbb{N}}
\DeclareMathOperator{\Borel}{\mathscr{B}}     
\DeclareMathOperator{\1}{\mathbbm 1}

\renewcommand{\O}{\mathcal{O}}   
\newcommand{\D}{\mathcal{D}}      
\newcommand{\leb}{\text{\rm leb}} 

\title{Modelling L\'evy space-time white noises}
\date{9 July 2019}
\author{Matthew Griffiths}
\author{Markus Riedle}
\affil{Department of Mathematics \\ King's College  \\ London WC2R 2LS\\ United Kingdom}

\begin{document}
\maketitle
\begin{abstract}
Based on the theory of independently scattered random measures, we introduce a natural generalisation of Gaussian space-time white noise to a L\'evy-type setting, which we call L\'evy-valued random measures. We determine the subclass of cylindrical L\'evy processes which  correspond to  L\'evy-valued random measures, and describe the elements of this subclass uniquely by their characteristic function. We embed the L\'evy-valued random measure, or the corresponding cylindrical L\'evy process, in the space of general and tempered distributions. For the latter case, we show that this embedding is possible if and only if a certain integrability condition is satisfied. Similar to existing definitions, we introduce L\'evy-valued additive sheets, and show that integrating a L\'evy-valued random measure in space defines a L\'evy-valued additive sheet. This relation is manifested by the result, that  a L\'evy-valued random measure can be viewed as the weak derivative of a L\'evy-valued additive sheet in the space of distributions. 
\end{abstract}


\section{Introduction}

Gaussian random perturbations of partial differential equations are most often modelled either as a cylindrical Brownian motion or a Gaussian space-time white noise. The choice usually depends on the exploited method, in which one follows either a semi-group approach, based on the work by Da Prato and Zabczyk in \cite{DaPrato2014},  or a random field approach, originating from the work by Walsh in \cite{Walsh1986}. It is well known that both models essentially result in the same dynamics as established  by Dalang and Quer-Sardanyons in \cite{Dalang2011}. 

Cylindrical Brownian motions can be naturally generalised to cylindrical L\'evy processes by exploiting the theory of cylindrical measures and random variables. This was accomplished by one of us together with Applebaum in \cite{Applebaum2010}. In the random field approach, Gaussian space-time white noise is usually generalised to a L\'evy-type setting by utilising either a random measure approach or a L\'evy-It{\^o} decomposition. However, a variety of definitions are used for both approaches in the literature.  

The objectives of our work are the comparison of cylindrical L\'evy processes 
with L\'evy-type models in the random field approach, and the embedding of these models in the space of general and tempered (Schwartz) distributions. It is well known in the Gaussian setting, that the standard cylindrical Brownian motion corresponds to the standard Gaussian space-time white noise, see e.g.\ Kallianpur and Xiong \cite{Kallianpur1995}, and that Gaussian space-time white noise can be viewed as a tempered distribution, see e.g. Gel'fand and Vilenkin \cite{Gel'fand-Vilenkin}. We show in this work that  these results significantly differ in the L\'evy-type setting. 

For having a sufficiently general model but distinguishing the time domain, we suggest the definition of {\em L\'evy-valued random measures}, which covers many of the existing L\'evy-type models in the random field approach in the literature, and which is based on independently scattered (or completely) random measures as investigated  by Rajput and Rosinski in \cite{Rajput1989}. It can be seen that this definition naturally extends the usual definition of Gaussian space-time white noise; see Remark \ref{re.Gauss-and-Levy-white-noise}. We provide a rigorous formulation of the relation between L\'evy-valued random measures and models of L\'evy-type space-time noise assuming a L\'evy-It{\^o} decomposition. 

We show that L\'evy-valued random measures correspond to members of a subspace of cylindrical L\'evy processes, and vice-versa. Because the restriction to a subspace is due to the assumed property of independent scattering for  random measures, we introduce the analogue property for cylindrical L\'evy processes. In contrast to the Gaussian case, where essentially only the standard cylindrical Brownian motion is independently scattered, the subspace of independently scattered cylindrical L\'evy processes is varied.  We completely characterise the sub-class of independently scattered cylindrical L\'evy processes by their particular form of their characteristic function. 

Furthermore, we introduce L\'evy-valued additive sheets based on the classical work by Adler et al \cite{Adler1983} and its extension by Dalang and Humeau \cite{Dalang2015}. We establish the relation between L\'evy-valued random measures and additive sheets, which is given by the integration of the L\'evy-valued random measure in space.  This relation is established to be one-to-one for  L\'evy-valued random measures without fixed point of discontinuity in space.

We embed L\'evy-valued random measures and, due to the aforementioned correspondence, independently scattered cylindrical L\'evy processes, in the space of distributions and tempered distributions. Although the embedding in the former case is possible for all L\'evy-valued random measures, L\'evy-valued random measures are tempered distributions if and only if they satisfy a certain integrability condition. In both cases, we show that the embedded cylindrical L\'evy process induces a genuine L\'evy process in the space of general or tempered distributions, i.e.\ a regularisation in the sense of It{\^o} and Nawata \cite{Ito-Nawata}.  

Embedding L\'evy-valued random measures in the space of distributions relates them to a model of L\'evy white noise in the space of distributions, investigated in a series of papers by Dalang, Humeau, Unser and co-authors in  e.g.\ \cite{Aziznejad2018,Dalang2015,Fageot2016,Fageot2017a}. Their model of noise is initiated from research on developing sparse statistical models for signal and image processing. However, the two models result in the same object only under the restrictive assumption that the L\'evy-valued random measure is stationary  in space. 
Similar questions as in our work are addressed in \cite{Dalang2015}  and  \cite{Fageot2017a} for the L\'evy white noise in the space of distributions. 

Combining our results with the work \cite{Aziznejad2018} by Aziznejad, Fageot and Unser enables us to determine the Besov space, in which the paths of a cylindrical L\'evy process attain values. This conclusion is so far only possible for independently scattered L\'evy processes which are stationary in space, due to the assumptions in \cite{Aziznejad2018}. Nevertheless, it indicates a potential reasoning for the often observed phenomena of irregular trajectories of solutions of heat equations driven by cylindrical L\'evy processes, e.g.\ in Brze\'zniak and Zabczyk \cite{Brzezniak2010} and Priola and Zabczyk \cite{Priola2011}; see Example  \ref{ex.alpha-in-Besov}.

Our article starts with Section \ref{subsec:IDRM} on L\'evy-valued random measures. Here, we introduce the definition, provide its relation to models described by a L\'evy-It{\^o} decomposition, and derive its embedding in the space of general and tempered distributions. In Section \ref{se.additive-sheets}, we introduce L\'evy-valued additive sheets and establish their correspondence to L\'evy-valued random measures. Finally, in the last section, we provide the relation between cylindrical L\'evy processes and L\'evy-valued random measures, and investigate the subspace of independently scattered cylindrical L\'evy processes. \\

{\bf Notation: } for a Borel set $\O\subseteq \R^d$ we denote the Borel $\sigma$-algebra by $\Borel(\O)$  and define the $\delta$-ring
\begin{align*}
\Borel_b(\O):=\left\{A\in\Borel(\O)\colon A \text{ is relatively compact}\right\}.
\end{align*}
The Lebesgue measure on $\Borel(\R^d)$ is denoted by $\leb$. The closed unit ball in $\R^d$ is denoted by $B_{\R^d}$. 

Throughout the paper, we fix a probability space $(\Omega,{\mathcal A},P)$. The space of $P$-equivalence classes of 
measurable functions $f\colon \Omega\to\R$ is denoted by $L^0(\Omega, P)$, and of 
$p$-th integrable functions by $L^p(\Omega, P)$ for $p>0$. These spaces are equipped with their standard metrics and (quasi-)norms. 


\section{L\'evy-valued random measures}
\label{subsec:IDRM}

Our definition of L\'evy-valued random measures is based on the work \cite{Rajput1989}
by Rajput and Rosinski. Instead of general $\delta$-rings, it is sufficient for us to restrict ourselves to the $\delta$-ring $\Borel_b(\O)$ of all relatively compact subsets of the Borel set $\O\in\Borel(\R^d)$ as the domain of the random measures.  
\begin{definition}
	\label{def:LevyTypeMeasure}
	A map $M\colon \Borel_b(\O)\to L^0(\Omega,P)$ is called an {\em independently scattered random measure} on $\Borel_b(\O)$ if for each collection of disjoint sets $A_1,A_2,\ldots \in\Borel_b(\O)$ the following hold:
	\begin{enumerate}
	\item[\rm{(a)}]  the random variables $M(A_1),\, M(A_2),\ldots$ are independent;
	\item[\rm{(b)}] if $\displaystyle\bigcup_{k\in\mathbb{N}}A_k \in\Borel_b(\O)$ then 
		$\displaystyle 
		M\left(\bigcup_{k\in\mathbb{N}}A_k\right) = \sum_{k\in\mathbb{N}}M(A_k)\;\; P$-a.s.
	\end{enumerate}
	An independently scattered random measure $M$ is called {\em infinitely divisible} 
	if 
	\begin{enumerate}
	\item[{\rm (c)}] the random variable $M(A)$ is infinitely divisible for each $A\in\Borel_b(\O)$. 
	\end{enumerate}
Analogously, an independently scattered random measure is called {\em Gaussian} (or {\em Poisson}), 
if $M(A)$ is Gaussian (or Poisson) distributed for each $A\in \Borel_b(\O)$. 
\end{definition}

In this setting, it is shown in \cite{Rajput1989} that there exist 
\begin{enumerate}
\item[{\rm (1)}] a signed measure 
	$\gamma\colon \Borel_b(\O) \to \R$,
\item[{\rm (2)}]  a measure
	$\Sigma\colon \Borel_b(\O) \to \Rp$,
\item[{\rm (3)}] a $\sigma$-finite measure
	$\nu\colon \Borel(\O\times\R) \to [0,\infty]$,
\end{enumerate}
such that for each $A\in\Borel_b(\O)$ the characteristics of $M(A)$ is given by $(\gamma(A), \Sigma(A), \nu_{A})$, where the L\'evy measure $\nu_A$ on $\Borel(\R)$ is defined by $\nu_A(\cdot):=\nu(A\times \cdot)$. We call the triple $(\gamma,\Sigma,\nu)$ the \textit{characteristics of $M$}. Furthermore, we may extend the total variation $\norm{\gamma}_{{\rm TV}}$ of $\gamma$ and $\Sigma$ to $\sigma$-finite measures on $\Borel(\O)$. In this case, 
the mapping 
\[
	\lambda\colon\Borel(\O)\to [0,\infty],\qquad     
	\lambda(A) = \norm{\gamma}_{{\rm TV}}(A) + \Sigma(A) + \int_{\mathbb{R}} (\abs{y}^2 \wedge 1) \,\nu(A,\mathrm{d}y),
\]
defines a $\sigma$-finite measure, which is called the {\em control measure of $M$}. We note that $\lambda(A)<\infty$ for $A\in\Borel_b(\O)$. The control measure $\lambda$ is called {\em atomless} if $\lambda(\{x\})=0$ for all $x\in\O$.

We extend Definition \ref{def:LevyTypeMeasure} to include a dynamical aspect, i.e.\ a time variable. This extension can be thought of as a similar construction to that of Walsh in \cite{Walsh1986}.
\begin{definition}
	\label{def:LWNRM}
	A family $(M(t):t\geq 0)$ of infinitely divisible random measures $M(t)$ on $\Borel_b(\O)$ is called a {\em L\'evy-valued random measure on $\Borel_b(\O)$} if, for every $A_1,\dots, A_n\in\Borel_b(\O)$ and $n\in\N$, the stochastic process 
	\begin{align*}
		\big((M(t)(A_1),\dots, M(t)(A_n)):\, t\ge 0\big)
	\end{align*}
	is a L\'evy process in $\R^n$. We shall write $M(t,A):=M(t)(A)$.
\end{definition}

Let $(M(t):t\geq 0)$ be a L\'evy-valued random measure on $\Borel_b(\O)$, and suppose $(\gamma,\Sigma,\nu)$ and $\lambda$ are the characteristics and control measure, respectively, of the infinitely divisible random measure $M(1)$. Then, it follows from the stationary increments of the process $(M(t,A):\, t\ge 0)$  that for each $t\geq 0$ the characteristics of the infinitely divisible random measure $M(t,A )$ is given by $(t\gamma,t\Sigma,t\nu)$, and the control measure of $M(t)$ is given by $t\lambda$. We shall refer to $(\gamma,\Sigma,\nu)$ as the {\em characteristics  of $M$} and $\lambda$ as the {\em control measure of $M$}. 

\begin{proposition}\hfill
\label{prop:LVRMisISRM}
	\begin{enumerate}
	\item[{\rm (a)}] Let $M=(M(t):t\geq 0)$ be a L\'evy-valued random measure on $\Borel_b(\O)$. Then, there exists a unique infinitely divisible random measure $M'$ on $\Borel_b(\Rp\times \O)$ such that
		\begin{align*}
			M'((0,t]\times A)=M(t,A) \qquad \text{for each } t> 0,\,A\in \Borel_b(\O).
		\end{align*}
	\item[{\rm (b)}] Each infinitely divisible random measure $M'$ on 
		$\Borel_b(\Rp\times \O)$ with control measure $\lambda=\leb\otimes \lambda_0$ 
		for a $\sigma$-finite measure $\lambda_0$ on $\Borel(\O)$ defines by 
		\begin{align*}
			M(t,A):=M'((0,t]\times A)
			\qquad\text{for each } t> 0,\,A\in \Borel_b(\O),
		\end{align*}
		a L\'evy-valued random measure on $\Borel_b(\O)$. 
	\end{enumerate}
\end{proposition}

\begin{proof}
(a)	Let $\mathcal{R}$ be the ring consisting of finite unions of disjoint sets of the form $(s,t]\times A$ for some $0\le s\le t$ and $A\in\Borel_b(\O)$. We define a process $M'$ on $\mathcal{R}$ by the prescription
	\begin{align*}
		M'\left(\bigcup_{i=1}^n (s_i,t_i]\times A_i\right):=\sum_{i=1}^n\big(M(t_i,A_i)-M(s_i,A_i)\big)
		\,.
	\end{align*}
For a set $B=(s,t]\times A$ in  $\mathcal{R}$ the characteristic exponent $\Psi$ of 
the random variable $M'((s,t]\times A)= M(t,A)-M(s,A)$ can be estimated for all $u\in \R$ by
\begin{align*}
\abs{\Psi(u)}&\le (t-s)
\Big(\abs{u}\norm{\gamma}(A)+\tfrac{1}{2}u^2\Sigma(A)+\tfrac{1}{2}u^2\int_{A\times B_{\R}}y^2\,\nu({\rm d}x,{\rm d}y)+2\nu(A\times B_{\R}^c)\Big)\\
&= \abs{u} \big( \leb\otimes\norm{\gamma}\big)(B)+ 
\tfrac12 u^2 \big( \leb\otimes\Sigma\big)(B)\\
&\qquad + \tfrac12 u^2 \int_{B\times B_{\R}} y^2 \, \big(\leb\otimes \nu\big)({\rm d}v,{\rm d}x,{\rm d}y \big) + 2 \big(\leb\otimes \nu\big)(B\times B_{\R}^c).
\end{align*}
Since the last line is finitely additive for disjoint sets, the estimate above extends to 
arbitrary sets $B=(s_1,t_1]\times A_1 \cup\cdots \cup (s_n,t_n]\times A_n\in \mathcal{R}$. 
Thus, if $(B_n)_{n\in\N}\subseteq \mathcal{R}$ is a sequence with $B_n\downarrow\emptyset$
we obtain $M'(B_n) \to 0$ in probability. 
Consequently, we may apply Theorem 2.15 in \cite{Kallenberg2017}, considering positive and negative parts separately, to extend $M'$ to a
unique independently scattered random measure on $\Borel_b(\O)$. Infinite divisibility of the extension follows immediately from that of $M$.

(b) Follows immediately from the assumed form of the control measure and the independent scattering of $M'$. 
\end{proof}

\begin{remark}\label{re.Gauss-and-Levy-white-noise}
Gaussian space-time white noise is usually defined equivalently to a Gaussian random measure on $\Borel_b(\Rp\times \O)$ in the sense of Definition \ref{def:LevyTypeMeasure}. Typically, one assumes that the measure $\Sigma$ on $\Borel_b(\Rp\times \O)$ is either the Lebesgue measure or of the form $\Sigma=\leb\otimes \Sigma_0$ for a $\sigma$-finite measure $\Sigma_0$ on $\Borel_b(\O)$.  Thus, Part (b) of Proposition \ref{prop:LVRMisISRM} shows that our definition of a L\'evy-valued random measure naturally extends the class of Gaussian space-time white noises to a L\'evy-type setting. 
\end{remark}

The relation between random measures and models of L\'evy-type noise utilising a 
L\'evy-It{\^o} decomposition seems to be well known. We rigorously formulate this result in our setting:
\begin{proposition}
\label{prop:LRM}
	Let $\zeta$ be a $\sigma$-finite Borel measure on $\Borel(\O)$ and $(U,\mathcal{U},\nu)$ a  $\sigma$-finite measure space. Assume that 
	\begin{enumerate}
		\item[{\rm (a)}] $\rho\colon \Borel_b(\O)\to\R$ is a signed measure;
		\item[{\rm (b)}] $W\colon \Borel_b(\Rp\times\O)\to L^2(\Omega,P)$ is a  Gaussian random measure with characteristics $(0,\leb\otimes\zeta,0)$; 
		\item[{\rm (c)}] $N\colon \Borel_b(\Rp\times\O)\otimes\mathcal{U}\to L^0(\Omega,P)$ is Poisson random measure  with intensity $\leb\otimes\zeta\otimes\nu$, independent of $W$, and with compensated Poisson random measure $\widetilde{N}$.
	\end{enumerate}
	Then for any  functions
	\begin{enumerate}
		\item[{\rm (1)}] $b\in L^2(\O,\zeta)$, 
		\item[{\rm (2)}] $c\colon \O \times U \to \R$ with 
		 $\int_{\O\times U} \left( \abs{c(x,y)}^2 \wedge \abs{c(x,y)}\right)\,(\zeta\otimes\nu)(dx,dy)<\infty$, 
		\item[{\rm (3)}] $d:\O\times U\to \mathbb{R}$ with 
		  $\int_{\O\times U} \big(\abs{d(x,y)} \wedge 1\big) \,(\zeta\otimes\nu)(dx,dy)<\infty$,
	\end{enumerate}
	we define a mapping $M'\colon\Borel_b(\Rp\times\O)\to L^0(\Omega,P)$ by
	\begin{align*}
		M'(B) &= \big(\leb\otimes\rho\big)(B) + \int_{B} b(x)\,  W(\mathrm{d}s, \mathrm{d}x) \\
			&\qquad	 + \int_{B\times U} c(x,y)\,\widetilde{N}(\mathrm{d}s,\mathrm{d}x,\mathrm{d}y)
				+ \int_{B\times U} d(x,y)\,N(\mathrm{d}s,\mathrm{d}x,\mathrm{d}y)
	\,.\end{align*}
	Then we obtain a L\'evy-valued random measure on $\Borel_b(\O)$ by the prescription
	\[M(t,A) := M'((0,t]\times A)
	\qquad\text{for all }A\in\Borel_b(\O),\,t\ge 0.
	\]
\end{proposition}

\begin{proof}
	The existence of the Gaussian integral is guaranteed by \cite{Walsh1986} and that of the Poisson integrals  by \cite[Lemma\ 12.13]{Kallenberg2002}. We first show that $M'$ forms an infinitely divisible random measure. The $\sigma$-additivity 
	and independent scattering follow immediately from the definition of the stochastic integrals. 
	It follows from Proposition 19.5 in \cite{sato2013} that for fixed $t\in \Rp$ and $A\in\Borel_b(\O)$ the characteristic function of the random variable $M'((0,t]\times A)$ is given by 
	\begin{align*}
		&\phi_{M'(0,t]\times A)}(u) = 	
		\exp\left(t\left( iu\rho(A) - \tfrac{1}{2}u^2\int_A b^2(x) \,\zeta(\mathrm{d}x)   \right.\right. \\
		&+ \!\! \left.\left.\int_A\int_{U} \!\left(e^{iuc(x,y)}-1-iuc(x,y)\right) \nu(\mathrm{d}y)\,\zeta(\mathrm{d}x) 
		+ \int_A\int_{U} \!\left(e^{iud(x,y)}-1\right) \nu(\mathrm{d}y)\, \zeta(\mathrm{d}x) \right)\right),
	\end{align*}
	which shows infinite divisibility; furthermore the control measure is seen to be proportional to Lebesgue measure in $t$. Consequently, by Proposition \ref{prop:LVRMisISRM}, we obtain the L\'evy-valued random measure $M$. 	
\end{proof}

\begin{example}\label{ex.Balan}
	Balan \cite{Balan2014} defines $\alpha$-stable L\'evy noise for $\alpha \in (0,2)$ as a random measure $M'$ given by, for bounded sets $B$ in $\Borel_b(\mathbb{R}_+\times\mathbb{R}^d)$,
	\[M'(B):=
		\begin{cases}
\int_{B\times\mathbb{R}} y \,N(\mathrm{d}s,\mathrm{d}x,\mathrm{d}y),&\text{if }\alpha\in(0,1],\\
		\int_{B\times\mathbb{R}} y \,\widetilde{N}(\mathrm{d}s,\mathrm{d}x,\mathrm{d}y),&\text{if } \alpha\in(1,2),
\end{cases}
	\]
	where $N$ is a Poisson random measure on $\Borel_b(\mathbb{R}_+\times\mathbb{R}^d\times\mathbb{R})$ with intensity $\leb\otimes\leb\otimes\nu_{\alpha}$, and
	\[\nu_{\alpha}(\mathrm{d}y) =
		\left(p\alpha y^{-\alpha-1}\mathbbm{1}_{(0,\infty)}(y) 
		+q\alpha (-y)^{-\alpha-1}\mathbbm{1}_{(-\infty,0)}(y)\right) \mathrm{d}y
	\]
	for some $p+q=1$ (for the case $\alpha=1$ it is required that $p=q=\tfrac{1}{2}$). Proposition \ref{prop:LRM} guarantees that, by defining $M(t,A):=M'((0,t]\times A)$ for $t\geq 0$ and $A\in\Borel_b(\R^d)$, we obtain a L\'evy-valued random measure $M$ on $\Borel_b(\R^d)$. Direct calculation shows that the characteristic function of $M(t,A)$ is given by, for $t\geq 0$, $A\in\Borel_b(\R^d)$ and $u\in\R$,
	\[\varphi_{M(t,A)}(u) = \exp \Bigg(t\cdot\leb(A)\cdot \Big(
		i\beta\frac{\alpha}{1-\alpha}u + \int_{\R} \big(e^{iuy}-1-iuy\1_{B_{\R}}(y)\big)\,\nu_{\alpha}({\rm d}y)
		\Big)\Bigg)
	\,,\]
	where $\beta:=p-q$, and thus we see the characteristics of $M$ are $\big(\beta\tfrac{\alpha}{1-\alpha}\leb,0,\leb\otimes\nu_{\alpha}\big)$. The control measure is given by
	\[\lambda(A) = \Big(\abs{\beta\frac{\alpha}{1-\alpha}}+\frac{2}{2-\alpha}\Big)\leb(A)\qquad\text{for } A\in\Borel(\R^d),\]
	with the necessary restriction for the case $\alpha=1$.
\end{example}

\begin{example}
	Mytnik, in \cite{Mytnik2002}, considers a martingale-valued measure $(M(t,A):t\geq 0,A\in\Borel_b(\R^d))$ in the sense of Walsh \cite{Walsh1986},  such that for any $A\in\Borel_b(\R^d)$, the process $(M(t,A):t\geq 0)$ is a real-valued $\alpha$-stable process ($\alpha\in(1,2)$), with Laplace transform
	\[E\left[e^{-uM(t,A)}\right]
		=	e^{-tu^{\alpha}\cdot\leb(A)},
		\qquad t\geq 0,u\geq 0.
	\]
	The author terms $M$ an \emph{$\alpha$-stable measure without negative jumps}.
\end{example}

\begin{example}
	Basse-O'Connor and Rosinski in Section 4 of \cite{OConnor2016} consider an infinitely divisible random measure $M$ on $\R\times V$, for some countably-generated measure space $V$, which is invariant under translations over $\R$. By Proposition \ref{prop:LVRMisISRM}, $M$ defines a L\'evy-valued random measure on $V$.
\end{example}

Because of the multiplicative relation between the characteristics of the infinitely divisible random measures $M(1)$ and $M(t)$, remarked after Definition \ref{def:LWNRM}, the integration theory for infinitely divisible random measures developed in \cite{Rajput1989}  directly extends to L\'evy-valued random measures $(M(t):t\geq 0)$ on $\Borel_b(\O)$. For a simple function 
\begin{align}\label{eq.def-simple}
	f\colon \O \to \R,  \qquad
	f(x)= \sum_{k=1}^n \alpha_k\mathbbm{1}_{A_k}(x),
\end{align}
for $\alpha_k\in\R$ and pairwise disjoint sets $A_1,\dots, A_n\in \Borel_b(\O)$, 
the integral is defined as 
\begin{align}\label{eq.def-intM-simple}
	\int_A f(x)\,M(t, \mathrm{d}x)
		:= \sum_{k=1}^n \alpha_k M(t,A\cap A_k)
	\qquad\text{for all } A\in \Borel(\O),\, t\ge 0. 
\end{align}
An arbitrary measurable function $f\colon \O\to\mathbb{R}$ is said to be \textit{$M$-integrable} if the following hold:
\begin{enumerate}
	\item[(1)] there exists a sequence of simple functions $(f_n)_{n\in\N}$ of the form \eqref{eq.def-simple} such that $f_n$ converges pointwise to $f$ $\lambda$-a.e., where $\lambda$ is the control measure of $M$;
	\item[(2)] for each $A\in\Borel(\O)$ and $t\geq 0$, the sequence $\big(\int_A f_n(x)\,M(t,\mathrm{d}x) \big)_{n\in\N}$ converges in probability.
\end{enumerate}
In this case, we define 
\begin{align}\label{eq.def-intM}
\int_A f(x)\,M(t, \mathrm{d}x) := P\mathrm{-}\lim_{n\to\infty}  \int_A f_n(x)\,M(t,\mathrm{d}x).
\end{align}
It is clear, by the stationarity of the increments of L\'evy processes, that Condition (2) above holds for all $t\geq 0$ if it holds for any $t>0$. Furthermore, Theorem 3.3 in \cite{Rajput1989} identifies the set of $M$-integrable functions as the Musielak-Orlicz space 
\[L_{M}(\O,\lambda) := \left\{f\in L^0(\O,\lambda):\int_{\O} \Phi_M(\abs{f(x)},x)\,\lambda(\mathrm{d}x)<\infty \right\},
\] 
where $\Phi_M:\mathbb{R}\times \O\to\mathbb{R}$ is defined as:
\begin{align}\label{eq.def-Phi-M}
\Phi_M(u,x)
	:= \sup_{\abs{c}\leq 1} \abs{U(cu,x)}& + u^2g(x) + \int_{\mathbb{R}} \left(1\wedge\abs{uy}^2\right)\,\rho(x,\mathrm{d}y),\\
\text{with }	U(u,x)&:=ua(x)+u\int_{\mathbb{R}} y\big(\1_{B_{\R}}(uy)-\1_{B_{\R}}(y)\big)\,\rho(x,\mathrm{d}y),\notag \\
	 a(x)&:=\frac{\mathrm{d}\gamma}{\mathrm{d}\lambda}(x),\qquad
	g(x)=\frac{\mathrm{d}\Sigma}{\mathrm{d}\lambda}(x).\notag
\end{align}
Here, $(\gamma,\Sigma,\nu)$ denotes the characteristics of $M$. 
The  measure $\rho(x,\cdot)$ is a disintegration of $\nu$ over $\lambda$, i.e.\  $\int_{\O\times\mathbb{R}} h(x,y)\,\nu(\mathrm{d}x,\mathrm{d}y) = \int_{\O}\Big(\int_{\mathbb{R}} h(x,y)\,\rho(x,\mathrm{d}y)\Big)\,\lambda(\mathrm{d}x)$ for each measurable function $h\colon \O\times\R\to \Rp$. 
The space $L_M(\O,\lambda)$ 
is a complete, translation-invariant,  linear metric space.
Furthermore for all $t\geq 0$, the mapping 
\begin{align}\label{eq:JLMtoL0}
	J(t):L_M(\O,\lambda)\to L^0(\Omega,P),
	\qquad
	J(t)f=\int_{\O} f(x)\,M(t,{\rm d}x),
\end{align}
is continuous.
Finally, Proposition 2.6 in \cite{Rajput1989} allows us to immediately state the L\'evy symbol of $J(\cdot)f$ as, for $u\in\R$,
\begin{align}
\label{eq:MIntChars}
	\Psi_{J(\cdot)f}(u) = 
		& iu\int_{\O} f(x)\,\gamma({\rm d}x) - \tfrac{1}{2}u^2\int_{\O} f^2(x)\,\Sigma({\rm d}x)\notag \\
		& +\int_{\O\times\R}\Big(e^{iuf(x)y}-1-iuf(x)y\1_{B_{\R}}(y)\Big)\,\nu({\rm d}x,{\rm d}y)
	\,.
\end{align}

For an open set $\O\subseteq \R^d$ let $\mathcal{D}(\O)$ denote the space of infinitely differentiable functions with compact support. We equip $\D(\O)$ with the inductive topology, that is, the strict inductive limit of the Fr\'echet spaces $\mathcal{D}(K_i):=\{f\in C^{\infty}(\R^d):\text{supp}(f)\subseteq K_i\}$ where $\{K_i\}_{i\in\N}$ is a strictly increasing sequence of compact subsets of $\O$ such that $\O=\bigcup_{i\in\N}K_i$. The topological dual space $\D^\ast(\O)$ is called the space of distributions, which we equip with the strong topology, that is the topology generated by the family of seminorms $\{\eta_B\}$, where for each bounded $B\subseteq \D(\O)$ we define $\eta_B(f):=\sup_{\phi\in B}\abs{\scapro{\phi}{f}}$ for $f\in\D^{\ast}(\O)$. In these topologies $\D(\O)$ is reflexive \cite[Page~376]{Treves1967}.

Analogously as locally integrable functions and measures are identified with distributions, we proceed to relate a L\'evy-valued random measure $M$ on $\Borel_b(\O)$ to a distribution-valued process. For this purpose, we define for each $t\ge 0$ the integral mapping 
\begin{align}\label{eq.def-J}
	J_{\D}(t):\D(\O)\to L^0(\Omega,P)\,,
	\qquad
	J_{\D}(t)f=\int_{\O} f(x)\,M(t,{\rm d}x).
\end{align}
In the proof of Theorem \ref{th.J-in-D} below we show that $\D(\O)$ is continuously embedded in $L_{M}(\O, \lambda)$, and thus the mapping $J_{\D}(t)$ is well-defined. 

\begin{theorem}\label{th.J-in-D}
	For a L\'evy-valued random measure $M$ on $\Borel_b(\O)$ let $J_{\D}$ be defined by \eqref{eq.def-J}. Then there exists a genuine L\'evy process $(Y(t):\, t\ge 0)$ in $\D^\ast(\O)$ satisfying
	\begin{align*}
		\scapro{f}{Y(t)}=J_{\D}(t)f \qquad\text{for all }f\in \D(\O), \, t\ge 0. 
	\end{align*}
\end{theorem}

Our proof of this Theorem relies on the following two Lemmas. 
\begin{lemma}
\label{lem:JnDimLevy}
	For a L\'evy-valued random measure $M$ on $\Borel_b(\O)$ let $J$ be defined by \eqref{eq:JLMtoL0}. Then, for any $f_1,\ldots ,f_n \in L_M(\O,\lambda)$ and $n\in\N$, we have	that
\[\big((J(t)f_1,\ldots,J(t)f_n):\, t\geq 0\big)\]
	is a L\'evy process in $\R^n$.
\end{lemma}

\begin{proof}
	Let $f_k$ for $k=1,\dots, n$ be simple functions of the form
		\begin{align}\label{eq.simple-functions}
			f_k\colon \O \to\R, \qquad
			f_k(x)= \sum_{j=1}^{m_k} \alpha_{k,j}\mathbbm{1}_{A_{k,j}}(x),
		\end{align}
	for $\alpha_{k,j}\in\R$ and $A_{k,j}\in \Borel_b(\O)$
	with $A_{k,1},\dots,A_{k,m_k}$ disjoint for each $k\in\{1,\dots, n\}$. By taking the intersections of all possible permutations of the sets $A_{k,j}$, we can assume that 
	\begin{align*}
	f_k(x)=\sum_{j=1}^{m} \tilde{\alpha}_{k,j}\1_{\tilde{A}_{j}}(x)
	\quad\text{for all } x\in \O,
	\end{align*}
	for all $k=1,\dots, n$, where $\tilde{\alpha}_{k,j}\in\R$ and disjoint sets $\tilde{A}_1,\dots, \tilde{A}_m\in \Borel_b(\O)$ for some $m\in\N$. For each  $0\leq t_1<\cdots <t_n$ we obtain by the definition in 
		\eqref{eq.def-intM-simple} that
		\begin{align*}
			 J(t_1)f_1 &= \sum_{j=1}^{m} \tilde{\alpha}_{1,j} M\big(t_1,\tilde{A}_{j}\big), \\
			(J(t_2)-J(t_1))f_2 &= \sum_{j=1}^{m} \tilde{\alpha}_{2,j} \big(M(t_2,\tilde{A}_{j})-M(t_1,\tilde{A}_{j})\big), \\
					 \vdots \quad & = \quad\vdots \\
				 (J(t_n)-J(t_{n-1}))f_n& = \sum_{j=1}^{m} \tilde{\alpha}_{n,j} \big(M(t_n,\tilde{A}_{j})-M(t_{n-1},\tilde{A}_{j})\big).
		\end{align*}
	Independent increments of the L\'evy process $\big(M(\cdot,\tilde{A}_1),\ldots,M(\cdot,\tilde{A}_m)\big)$ together with independence of $M(t,\tilde{A}_i)$ and $M(t,\tilde{A}_j)$ for all $i,j=1,\dots, m$ with $i\neq j$
	implies that the random variables
		\begin{align*}
			J(t_1)f_1,\big(J(t_2)-J(t_1)\big)f_2,\dots, \big(J(t_n)-J(t_{n-1})\big)f_n,
		\end{align*}
		are independent. This property extends to arbitrary functions $f_1,\dots, f_n\in L_M(\O,\lambda)$ by the definition of the integrals in \eqref{eq.def-intM} as a limit of the integral for simple functions. It follows that the $n$-dimensional stochastic process $ \big( (J(t)f_1,\dots, J(t)f_n): \, t\ge 0\big)$ has independent increments.

Furthermore, if $f$ is a simple function of the form \eqref{eq.def-simple} then  
\begin{align}\label{eq.repeat-def-intM}
	 J(t)f=\sum_{k=1}^n \alpha_k M(t,A_k)
\end{align}
is a L\'evy process as it is the sum of independent L\'evy processes $M(\cdot,A_k)$. Approximating an arbitrary function $f\in L_{M}(\O,\lambda)$ by a sequence of simple functions and passing to the limit in \eqref{eq.repeat-def-intM} shows that $J(\cdot)f$ is a L\'evy process.
	 
	Let $f_1,\dots, f_n$ be arbitrary functions in $L_M(\O,\lambda)$. As $J(\cdot)f$ has stationary increments it follows that $ \big( (J(t)f_1,\dots, J(t)f_n): \, t\ge 0\big)$ has stationary increments by linearity. Furthermore, for each $c>0$ we have
	\begin{align*}
	P\big(\abs{((J(t)f_1,\dots, J(t)f_n))}>c\big)
	&= P\big(\abs{J(t)f_1}^2+\dots + \abs{J(t)f_n}^2>c^2\big)\\
	&\le \sum_{k=1}^n P\left( \abs{J(t)f_k}^2\ge \tfrac{c^2}{n}\right),
	\end{align*}
	and thus the stochastic continuity of $J(\cdot)f$ implies that of $ \big( (J(t)f_1,\dots, J(t)f_n): \, t\ge 0\big)$. Consequently, the latter is verified as an $n$-dimensional L\'evy process.
\end{proof}

\begin{lemma}
\label{lem:L2Embedding}
Let $M$ be a L\'evy-valued random measure on $\Borel_b(\O)$ with finite control measure $\lambda$. Then $L^2(\O,\lambda)$ is continuously embedded into $L_M(\O,\lambda)$.
\end{lemma}

\begin{proof}
Denote the characteristics of $M$ by $(\gamma,\Sigma,\nu)$. 
Note, that for arbitrary $g\in L^p(\O,\lambda)$ and $p\in [1,2]$, we have
\begin{align}\label{eq.int-f-lambda}
\begin{split}
		\int_{\O} g(x)\,\lambda({\rm d}x)
		&=\int_{\O} g(x)\,\norm{\gamma}_{TV}({\rm d}x)
		+ \int_{\O} g(x)\,\Sigma({\rm d}x)\\
		&\qquad+ \int_{\O\times B_{\R}} g(x)\abs{y}^2\,\nu({\rm d}x,{\rm d}y)
		+ \int_{\O\times B_{\R}^c} g(x)\,\nu({\rm d}x,{\rm d}y).
\end{split} 
\end{align}
For each $f\in L^2(\O,\lambda)$, the definition of $\Phi_M$ in \eqref{eq.def-Phi-M} gives
\begin{align}		\label{eq:Phi0pt1}
&\int_{\O} \Phi_M(\abs{f(x)},x)\,\lambda(\mathrm{d}x)\\
& = \int_{\O} \sup_{\abs{c}\leq 1} \abs{U(c\abs{f(x)},x)}\,\lambda({\rm d}x) 
	+ \int_{\O} \abs{f(x)}^2\,\Sigma({\rm d}x) 
 + \int_{\O\times\R} \big(1\wedge\abs{f(x)y}^2\big)\,\nu({\rm d}x,{\rm d}y).
  \notag
\end{align}
It follows from \eqref{eq.int-f-lambda} that the second integral in  \eqref{eq:Phi0pt1} is bounded by $\norm{f}^2_{L^2(\O,\lambda)}$.
Regarding the last integral in  \eqref{eq:Phi0pt1}, we obtain from \eqref{eq.int-f-lambda} that
\begin{align}\label{eq.int-max-f}
&\int_{\O\times\R}\hspace*{-.3cm}  \big(1\wedge \abs{f(x)y}^2  \big) \,\nu({\rm d}x,{\rm d}y)
 \leq \norm{f}^2_{L^2(\O,\lambda)}+ \int_{\O\times B_{\R}^c}\hspace*{-.2cm}\big(1\wedge \abs{f(x)y}^2  \big) \,\nu({\rm d}x,{\rm d}y).
	\end{align}
Lemma 2.8 in  \cite{Rajput1989} yields for the first integral in \eqref{eq:Phi0pt1} the estimate 
\begin{align}\label{eq.int-sup-U}
 &\int_{\O} \sup_{\abs{c}\leq 1} \big|U(c\abs{f(x)},x)\big|\,\lambda({\rm d}x)	\notag \\
&\qquad\qquad \leq \int_{\O}  \big|U(\abs{f(x)},x)\big| \,\lambda({\rm d}x) + 
 8\int_{\O\times\R} \big(1\wedge\abs{f(x)y}^2\big)\,\nu({\rm d}x,{\rm d}y). 
\end{align}
In \eqref{eq.int-max-f} we have already obtained an upper bound for the second integral in \eqref{eq.int-sup-U}. For estimating the first integral in \eqref{eq.int-sup-U} we obtain from the definition of $U$ and \eqref{eq.int-f-lambda} that
\begin{align*}
&\int_{\O}\abs{U(\abs{f(x)},x)}\,\lambda({\rm d}x) \\
&\qquad\quad = \int_{\O} \abs{
		\abs{f(x)}\Big(a(x)+\int_{\R}y\big(\1_{B_{\R}}(\abs{f(x)}y)-\1_{B_{\R}}(y)\big)\,\rho(x,{\rm d}y)\Big)} \,\lambda({\rm d}x) \\
&\qquad\quad \leq \int_{\O}\abs{f(x)}\,\norm{\gamma}_{TV}({\rm d}x)
		+ \int_{\O\times\R} \abs{f(x)y}\big(\1_{B_{\R}}(\abs{f(x)}y)-\1_{B_{\R}}(y)\big)\,\nu({\rm d}x,{\rm d}y)\\
& \qquad\quad\leq \norm{f}_{L^1(\O,\lambda)}+\norm{f}^2_{L^2(\O,\lambda)} +   \int_{\O\times B_{\R}^c}\abs{f(x)y} \1_{B_{\R}}(\abs{f(x)}y)
\,\nu({\rm d}x,{\rm d}y).
\end{align*}
Together with \eqref{eq.int-max-f} and \eqref{eq.int-sup-U},  we obtain
from  \eqref{eq:Phi0pt1} that
\begin{align}\label{eq.summary-Phi}
&\int_{\O} \Phi_M(\abs{f(x)},x)\,\lambda(\mathrm{d}x)
\le \norm{f}_{L^1(\O,\lambda)}+11 \norm{f}^2_{L^2(\O,\lambda)}\\
& + 9\int_{\O\times\R} \big(1\wedge\abs{f(x)y}^2\big)\,\nu({\rm d}x,{\rm d}y)
+\int_{\O\times B_{\R}^c}\abs{f(x)y} \1_{B_{\R}}(\abs{f(x)}y)
\,\nu({\rm d}x,{\rm d}y).  \notag
\end{align}
Since \eqref{eq.int-f-lambda} yields
\begin{align*}
\int_{\O\times B_{\R}^c}\hspace*{-.2cm}\big(1\wedge \abs{f(x)y}^2  \big) \,\nu({\rm d}x,{\rm d}y)
 \leq \int_{\O\times B_{\R}^c}  \,\nu({\rm d}x,{\rm d}y)\leq \lambda(\O),  \\
 \int_{\O\times B_{\R}^c}\abs{f(x)y} \1_{B_{\R}}(\abs{f(x)}y)\,\nu({\rm d}x,{\rm d} y)
  \leq \int_{\O\times B_{\R}^c}  \,\nu({\rm d}x,{\rm d}y)\leq \lambda(\O),
\end{align*}
we obtain  $f\in L_M(\O,\lambda)$ from \eqref{eq.summary-Phi}. 

To show that the embedding is continuous, let $(f_n)$ converge to $0$ in $L^2(\O,\lambda)$. We firstly show that the functions $(x,y)\mapsto f_n(x)y$ converge to $0$ in $\nu_1$-measure where $\nu_1:=\restr{\nu}{\O\times B_{\R}^c}$. For given $\epsilon>0$ define $M_n:=\{(x,y)\in \O\times B_{\R}^c : \abs{f_n(x)y}\geq\varepsilon\}$. 
As $\nu_1$ is a finite measure, there exists a compact set $K\subseteq \O\times B_{\R}^c$ such that $\nu_1(\O\times B_{\R}^c\setminus K)<\tfrac{\varepsilon}{2}$. Let $C:=\sup\{\abs{y}:(x,y)\in K\}$. Since $(f_n)$ converges in $L^1(\O,\lambda)$ it follows from \eqref{eq.int-f-lambda} that $(f_n)$ converges to $0$ in 
$L^1(\O\times B_{\R}^c,\nu)$, and thus in $\nu_1$-measure. Consequently,  there exists $N\in\N$ such that, for $n\geq N$,
	\[\nu_1\big(\{(x,y)\in \O\times B_{\R}^c : \abs{f_n(x)}\geq\tfrac{\varepsilon}{C}\}\big) \le \tfrac{\varepsilon}{2}
	\,.\]
Since $M_n\cap K\subseteq \big\{(x,y)\in \O\times B_{\R}^c:\abs{f_n(x)}\geq\frac{\varepsilon}{C}\big\}$, we obtain
\begin{align*}
\nu_1(M_n)= \nu_1(M_n\cap K)+ \nu_1(M_n\setminus K)\le \tfrac{\epsilon}{2}+
 \tfrac{\epsilon}{2}
 \qquad\text{for all }n\ge N, 
\end{align*}
which shows the claim. 

Since $\nu_1$ is a finite measure,  Lebesgue's theorem for dominated convergence in $\nu_1$-measure implies 
	\begin{align*}
		\lim_{n\to\infty}\int_{\O\times B_{\mathbb{\R}}^c} \abs{f_n(x)y }\1_{B_{\R}}(\abs{f_n(x)}y) \,\nu(\mathrm{d}x,\mathrm{d}y)=0. 
	\end{align*}
Similar arguments show that
	\begin{align*}
		\lim_{n\to\infty}\int_{\O\times B_{\mathbb{\R}}^c} \big(1\wedge\abs{f_n(x)y}^2\big) \,\nu(\mathrm{d}x,\mathrm{d}y)=0. 
	\end{align*}
Consequently, it follows from \eqref{eq.summary-Phi} that $(f_n)$ converges in $L_M(\Phi,\lambda)$, which completes the proof. 
\end{proof}

\begin{proof}[Proof of Theorem \ref{th.J-in-D}]
We first show that the space $\D(K)$ is continuously embedded in $L_M(\O,\lambda)$ for each compact $K\subseteq \O$. Trivially, the space $\D(K)$  is continuously embedded in $L^{\infty}(K,\lambda)$. As $K\in\Borel_b(\O)$, the control measure $\lambda$ is finite on $K$, and it follows that $L^{\infty}(K,\lambda)$ is continuously embedded in $L^2(K,\lambda)$. The latter is continuously embedded in $L_M(K,\lambda)$ by Lemma \ref{lem:L2Embedding}. Because  whenever ${\rm supp}(f)\subseteq K$ we have
	\[\int_{\O} \Phi_M(\abs{f(x)},x)\,\lambda(\mathrm{d}x)
		= \int_{K} \Phi_M(\abs{f(x)},x)\,\lambda(\mathrm{d}x), 
	\]
it follows that  $\D(K)$ is continuously embedded in $L_M(\O,\lambda)$. As $\D(\O)$ is the inductive limit of $\{\D(K_i)\}_{i\in\N}$, we thus conclude that $\D(\O)$ is continuously embedded in $L_M(\O,\lambda)$.
	
Let $\iota\colon\D(\O)\to L_M(\O,\lambda)$ be the continuous embedding. Then the mapping  $J_{\D}(t)\colon\D(\O)\to L^0(\Omega,P)$ can be represented as $J_{\D}(t) = J(t) \circ \iota$ for each $t\ge 0$, showing that $J_{\D}(t)$ is continuous. Lemma \ref{lem:JnDimLevy} shows that $J_{\D}$ is a cylindrical L\'evy process in $\D^\ast(\O)$ as defined in  \cite[Definition\ 3.6]{Fonseca2017}. Furthermore, since $J_{\D}(t)$ is continuous, and $\D(\O)$ is nuclear \cite[Theorem~51.5]{Treves1967} and ultrabornological \cite[Page~447]{Narici2011}, Theorem 3.8 in \cite{Fonseca2017} implies the existence of the $\D^\ast(\O)$-valued L\'evy process $Y$. 
\end{proof}

Let $\mathcal{S}(\R^d)$ denote the Schwartz space on $\R^d$, that is
	\begin{align*}
		\mathcal{S}(\R^d)
		:=\big\lbrace
			f\in C^{\infty}(\R^d)\colon
			\norm{f}_{\mathcal{S}_k} < \infty \text{ for all } k\in\N
		\big\rbrace, 
	\end{align*}
where the seminorms $\norm{\cdot}_{\mathcal{S}_k}$, $k\in\N$ are defined by 
\begin{align*}
	\norm{f}_{\mathcal{S}_k}
		:= \max_{\abs{s}\leq k} \sup_{x\in\R^d} (1+\abs{x}^2)^k \abs{\partial^s f(x)}. 
\end{align*}
In particular, $\mathcal{S}(\R^d)$ is metrisable, and $f_n\to f$ in $\mathcal{S}(\R^d)$ means $\norm{f_n-f}_{\mathcal{S}_k}\to 0$ for each $k\in\N$. 
The dual space of $\mathcal{S}(\R^d)$ is the space $\mathcal{S}^\ast(\R^d)$ of tempered distributions.

Define for each $t\ge 0$ the integral mapping 
\begin{align}
\label{eq:J-Schwartz}
	J_{\mathcal{S}}(t):\mathcal{S}(\R^d)\to L^0(\Omega,P)\,,
	\qquad
	J_{\mathcal{S}}(t)f=\int_{\R^d} f(x)\,M(t,{\rm d}x).
\end{align}
Clearly, the mapping $J_{\mathcal{S}}(t)$ is only well defined if $\mathcal{S}(\R^d)$ is embedded in $L_M(\R^d,\lambda)$. The following theorem gives an equivalent condition for this. 
\begin{theorem}
\label{thm:SchwartzEmbedding}
	Let $M$ be a L\'evy-valued random measure on $\Borel_b(\R^d)$ with control measure $\lambda$.
Then the following are equivalent:
\begin{enumerate}
\item[{\rm (a)}] $\mathcal{S}(\R^d)$ is continuously embedded in $L_M(\R^d,\lambda)$;
\item[{\rm (b)}] there exists an $r>0$ such that the function $x\mapsto(1+\abs{x}^2)^{-r}$ is in $L_M(\R^d,\lambda)$. 
\end{enumerate}	
In this case, the mapping $J_{\mathcal{S}}(t)$ as defined in \eqref{eq:J-Schwartz} is well-defined and continuous for each $t\geq 0$. Furthermore, there exists a genuine L\'evy process $(Y(t):\, t\ge 0)$ in $\mathcal{S}^{\ast}(\R^d)$ satisfying
	\begin{align*}
		\scapro{f}{Y(t)}=J_{\mathcal{S}}(t)f \qquad\text{for all }f\in \mathcal{S}(\R^d), \, t\ge 0 . 
	\end{align*}
\end{theorem}

\begin{proof}
We begin by showing the implication (b) $\Rightarrow$ (a), for which we suppose there exists $r>0$ such that $x\mapsto\big(1+\abs{x}^2\big)^{-r}$ is in $L_M(\R^d,\lambda)$. For each $f\in\mathcal{S}(\R^d)$ there exists $K>0$ such that $(1+\abs{x}^2)^{r}\abs{f(x)}\leq K$ for all $x\in\R^d$. Since  $\Phi_M(\cdot,x)$ is monotone for each $x\in\R^d$ according to \cite[Lemma\ 3.1]{Rajput1989}, we have 
	\begin{align*}
		\Phi_M(\abs{f(x)},x) \leq \Phi_M(K(1+\abs{x}^2)^{-r},x)
		\qquad\text{for each }x\in\R^d, 
	\end{align*}
which implies  $f\in L_M(\R^d,\lambda)$.
	
Let $(f_n)_{n\in\N}\subseteq\mathcal{S}(\R^d)$ be a sequence converging to $0$ in $\mathcal{S}(\R^d)$. As the convergence is uniform in $x$, we have the existence of another $K>0$ such that $(1+\abs{x}^2)^{r}\abs{f_n(x)}\leq K$ for all $x\in\R^d$ and for all $n\in\N$. For fixed $x\in\R^d$ we have $\Phi_M(\abs{f_n(x)},x)\to\Phi_M(0,x)=0$ by continuity \cite[Lemma~3.1]{Rajput1989}, and as $\int_{\R^d}\Phi_M(K(1+\abs{x}^2)^{-r},x)\,\lambda({\rm d}x)<\infty$,  Lebesgue's theorem for dominated convergence implies
	\begin{align*}
		\int_{\R^d}\Phi_M(\abs{f_n(x)},x)\,\lambda({\rm d}x)\to 0
		\qquad\text{as }n\to\infty, 
	\end{align*}
which completes the proof of the implication (b) $\Rightarrow$ (a).

Conversely, suppose $\mathcal{S}(\R^d)$ is continuously embedded in $L_M(\R^d,\lambda)$.  Thus, the identity mapping $\iota\colon \mathcal{S}(\R^d)\to L_M(\R^d,\lambda)$ is continuous. Then, there exists a neighbourhood 
\begin{align*}
		U(0;k,\delta):=		\big\{ f\in\mathcal{S}(\R^d) : \norm{f}_{\mathcal{S}_k} < \delta\big\}, 
	\end{align*}
for some $k\in\N$ and $\delta>0$ such that $\iota$ maps $U(0;k,\delta)$ into the open unit ball of $L_M(\R^d,\lambda)$. Let $(f_n)_{n\in\N}\subseteq\mathcal{S}(\R^d)$ be any sequence such that $\norm{f_n}_{\mathcal{S}_k}\to 0$. Then, $(f_n)$ is eventually in $U(0;k,\delta)$ and thus $(\iota f_n)$ is eventually in the unit ball and so is bounded in $L_M(\R^d,\lambda)$. By Proposition 4 of  \cite[p. 41]{Jarchow1981} we have the continuity of $\iota$ in the  semi-norm $\norm{\cdot}_{\mathcal{S}_k}$, and thus we may extend $\iota$ by continuity to the completion of $\mathcal{S}(\R^d)$ in this semi-norm. We thus obtain the integrability condition by observing that the $C^\infty(\R^d)$ mapping $x\mapsto (1+\abs{x}^2)^r$ has finite semi-norm $\norm{\cdot}_{\mathcal{S}_k}$ for $r\leq -k$.

As in the proof of Theorem \ref{th.J-in-D}, an application of Lemma \ref{lem:JnDimLevy} and Theorem 3.8 in \cite{Fonseca2017} establish  the existence of the L\'evy process $Y$ in ${\mathcal S}(\R^d)$. 
\end{proof}

\begin{remark}\label{re.Kabanava2008}
In Kabanava \cite{Kabanava2008}, it is shown that a Radon measure $\zeta$ can be identified with a tempered distribution in $\mathcal{S}^\ast(\R^d)$ if and only if there is a real number $r$ such that $x\mapsto (1+\abs{x}^2)^r$ is integrable over $\R^d$ with respect to $\zeta$. Our condition for the mapping $J_{\mathcal S}$ in Theorem 
\ref{thm:SchwartzEmbedding} is analogue.  
\end{remark}

\begin{remark}\label{re.M-as-both}
By Proposition \ref{prop:LVRMisISRM}, we may also view the L\'evy-valued random measure $M$ as an infinitely divisible random measure $M'$ on $\Borel_b(\Rp\times \O)$, and define the integral mapping
\begin{align*}
	J'_{\D}\colon\D((0,\infty)\times \O)\to L^0(\Omega,P)\,,
	\qquad
	J'_{\D}f=\int_{(0,\infty)\times \O} f(t,x)\,M'({\rm d}t,{\rm d}x).
\end{align*}
Analogously to Theorem \ref{th.J-in-D} we obtain an infinitely divisible random variable $Y'$ in $\D^{\ast}((0,\infty)\times \O)$ satisfying
\begin{align}
\label{eq:AltDistribRV}
	\scapro{f}{Y'}=J'_{\D}f
	\qquad \text{for all } f\in\D((0,\infty)\times \O) \,.
\end{align}
Similarly, under the conditions of Theorem \ref{thm:SchwartzEmbedding}, 
we may consider the operator $J'$ on the space $\mathcal{S}(\Rp\times \R^{d})$. 
\end{remark}

\begin{remark}\label{re.Levy-white-noise}
In a series of papers, e.g.\ \cite{Aziznejad2018,Dalang2015,Fageot2016,Fageot2017a}, Dalang, Humeau, Unser and co-authors have studied the L\'evy white noise $Z$ defined as a distribution. Here, 
$Z$ is defined as a cylindrical random variable in $\D^\ast(\R^d)$, i.e.\ a linear and continuous mapping $Z\colon \D(\R^d)\to L^0(\Omega,P)$, with characteristic function 
\begin{align*}
\phi_Z\colon \D(\R^d)\to {\mathbb C}, \qquad \phi_Z(f)=\exp\left(\int_{\R^d} \psi\big(f(x)\big) \,{\rm d}x\right),
\end{align*}
where $\psi\colon \R\to {\mathbb C}$ is defined by
\begin{align}\label{eq.char-Z-Unser}
\psi(u):=ipu-\tfrac{1}{2}\sigma^2 u^2 +\int_{\R}\big(e^{iuy}-1-iuy\1_{B_{\R}}(y)\big)\,\nu_0({\rm d}y),
	\end{align}
for some constants $p\in\R$ and $\sigma^2\in\Rp$ and a L\'evy measure $\nu_0$ on $\R$. 

Let $M$ be a L\'evy-valued random measure on $\Borel_b(\R^d)$  with characteristics $(\gamma,\Sigma,\nu)$ and $J_{\D}(t)$ the corresponding operator  
defined in \eqref{eq.def-J} for $t\ge 0$. By comparing the L\'evy symbol in \eqref{eq:MIntChars} with \eqref{eq.char-Z-Unser} it follows that, for fixed $t\ge 0$, the mapping $J_D(t)$ is a L\'evy white noise in the above sense, if and only if
	\[\gamma= p\cdot \leb, \qquad \Sigma = \sigma^2\cdot\leb,
	\qquad \nu=\leb\otimes\nu_0, \]
for some $p\in\R$, $\sigma^2\in\Rp$ and a L\'evy measure $\nu_0$ on $\R$.  It follows that $M(t,A) \buildrel\mathscr{D} \over = M(t,B)$ for any sets $A,B\in\Borel_b(\R^d)$ with $\leb(A)=\leb(B)$. In this case, we call  $M$ to be  {\em stationary in space}. 

Dalang and Humeau have shown in \cite{Dalang2015} that a L\'evy white noise in $\D^\ast(\R^d)$ with L\'evy symbol \eqref{eq.char-Z-Unser} takes values in $\mathcal{S}^\ast(\R^d)$ $P$-a.s.\ if and only if 
\begin{align*}
		\int_{\R}\big(\abs{y}^{\epsilon}\wedge\abs{y}^2\big)\,\nu_0({\rm d}y) < \infty
		\quad \text{for some }\epsilon>0\,.
\end{align*}
This result is analogue to our Theorem \ref{thm:SchwartzEmbedding}. However, as L\'evy-valued random measures are not necessarily stationary in space, our condition is more complex. For example, even in the pure Gaussian case with characteristics $(0,\Sigma,0)$, the measure $\Sigma$ must be tempered; cf. Remark \ref{re.Kabanava2008}. 

Regularity of the L\'evy white noise $Z$ in terms of Besov spaces is studied in 
\cite{Aziznejad2018}. Their results can be applied to a L\'evy-valued random measure if it is additionally assumed to be stationary in space, i.e.\ which can be considered as a L\'evy white noise in the above sense. We illustrate such an application in the following example. 
\end{remark}

\begin{example}\label{ex.alpha-in-Besov}
	Let $M$ be the $\alpha$-stable random measure, $\alpha\in(0,2)$, described in Example \ref{ex.Balan}. For simplicity we consider the symmetric case, i.e.\ 
$p=q=\tfrac12$. As the characteristics of $M$ is given by $(0,0,\leb\otimes\nu_{\alpha})$, it follows that $M$ is stationary in space. Thus, for a fixed time $t\ge 0$, the mapping $J_{\D}(t)$ or, equivalently the random variable $Y(t)$, where $Y$ denotes the L\'evy process derived in Theorem \ref{th.J-in-D}, can be considered as a L\'evy white noise in $\D^\ast(\R^d)$; see Remark \ref{re.Levy-white-noise}. Furthermore, since 
$\int_{\R}\big(\abs{y}^{\epsilon}\wedge\abs{y}^2\big)\,\nu_{\alpha}({\rm d}y) < \infty$ for $\epsilon<\alpha$, we have that $Y(t)$ is in $\mathcal{S}^*(\R^d)$ $P$-a.s. By applying the results from \cite{Aziznejad2018} we obtain the following: for $p\in(0,2)\cup 2\N\cup\{\infty\}$ and for all $t\geq 0$, we have, almost surely:
	\begin{align}
	\label{eq:BesovIn}
&		\text{if  } \tau<d\Big(\frac{1}{\alpha}-1\Big) \text{  and  } \rho<-\frac{d}{p\wedge\alpha}
		\text{  then  } Y(t)\in B_p^{\tau}(\R^d,\rho),\\
	\label{eq:BesovOut}
&		\text{if  } \tau>d\Big(\frac{1}{\alpha}-1\Big) \text{  or  } \rho>-\frac{d}{p\wedge\alpha}
		\text{  then  } Y(t)\notin B_p^{\tau}(\R^d,\rho) \,,
	\end{align}
	where $B_p^{\tau}(\R^d,\rho)$ is the weighted Besov space of integrability $p$, smoothness $\tau$ and asymptotic growth rate $\rho$. Furthermore, a modification of $Y$ is a L\'evy process in any Besov space satisfying \eqref{eq:BesovIn}, 
since its characteristic function is continuous in $0$, guaranteeing  stochastic continuity. 
\end{example}

\section{L\'evy-valued additive sheets}\label{se.additive-sheets}

Just as the Brownian sheet is the generalisation of a Brownian motion to a multidimensional index set, additive sheets are defined as the corresponding generalisation of an additive process.  
Adler et al.\ \cite{Adler1983} first defined additive random fields on $\mathbb{R}^d$, and termed them `L\'evy processes' should they be stochastically continuous. In \cite{Dalang1992}, Dalang and Walsh discuss L\'evy sheets in $\R^2$. Additive fields with stationary increments are considered by Barndorff-Nielsen and Pedersen in \cite{Barndorff-Nielsen-Pedersen} and are called `homogeneous L\'evy sheets'. 
Herein we present our definition based on the deposition of Dalang and Humeau in \cite{Dalang2015} which extends \cite{Adler1983}; this is also similar to the presentation by Pedersen in \cite{Pedersen2003}.

For $a,b\in\mathbb{R}^d$ write $a\leq b$ if $a_j\leq b_j$ for all $j=1,\ldots, d$ and similarly $a<b$, and define boxes $(a,b] := \{s\in\mathbb{R}^d:a<t\leq b\}$ and $[a,b] := \{s\in\mathbb{R}^d:a\leq t\leq b\}$; $[a,b)$ and $(a,b)$ are defined mutatis mutandi.

For a function $f\colon \mathbb{R}^d\to \R$,  we define the increment of $f$ over $(a,b]$ for $a$, $b\in\R^d$ with $a< b$ by
\[\Delta_a^b f := \sum_{\epsilon_1=0}^1 \cdots \sum_{\epsilon_k=0}^1 (-1)^{\epsilon_1+\cdots +\epsilon_k} f\big(c_1(\epsilon_1),\ldots c_k(\epsilon_k)\big),
\]
where $c_j(0)=b_j$ and $c_j(1)=a_j$. For example, in the case $d=2$ we have $\Delta_a^b f = f(b_1,b_2) - f(b_1,a_2) - f(a_1,b_2) + f(a_1,a_2)$.

The c\`adl\`ag property is generalised to random fields in the following way: 
a function $f\colon \R^d\to \R$ has \textup{limits along monotone paths (lamp)} if for every $x\in\mathbb{R}^d$ and any sequence $(x_n)_{n\in\N}\subseteq \R^d$ converging to $x$ with either $x_{n,i}< x_j$ or 
$x_{n,j}\ge x_j$ for all $n\in{\mathbb N}$ and $j\in \{1,\dots, d\}$ 
where $x=(x_1,\dots, x_d)$ and $x_n=(x_{n,1},\dots, x_{n,d})$,
the limit $f(x_n)$ exists as $n\to\infty$ and furthermore $f$ is right-continuous if 
$f(x_n)\to f(x)$ as $n\to\infty$ for all sequences  with 
$x\le x_n$ for all $n\in{\mathbb N}$. We note that this property is a path-based property, and thus in contrast to random measures we define our sheets as mappings from $\R^d\times\Omega\to\R$.

\begin{definition}
	\label{def:AddSheet}
	Let $I\subseteq\R^d$ with $0\in I$. 
	A real-valued stochastic process $(X(x):\, x\in I)$ is called an \textup{additive sheet} if the following conditions are satisfied:
	\begin{enumerate}
		\item[{\rm (a)}]  $X(x)=0$ a.s.\ for all $x=(x_1,\dots,x_d)\in I$
		with $x_j=0$ for some $j\in\{1,\dots, d\}$;
		\item[{\rm (b)}]  
		$\Delta_{a_1}^{b_1}X, \ldots ,\Delta_{a_n}^{b_n}X$ are independent
		for disjoint boxes $(a_1,b_1],\ldots ,(a_n,b_n]\subseteq I$;
		\item[{\rm (c)}]  $X$ is continuous in probability;
		\item[{\rm (d)}]  almost all sample paths of $X$ have limits along monotone paths and are right-continuous. 
	\end{enumerate}
\end{definition}

\begin{remark}
For relaxing the requirements in Definition \ref{def:AddSheet} we refer to \cite{Adler1983}, e.g. to capture arbitrary initial conditions or sheets which are not continuous in probability. In particular, it is shown that Conditions (a) -- (c) gurantee the  existence of a lamp and right-continuous modification.
\end{remark}

If $(X(x):\, x\in I)$ is an additive sheet then for fixed $x\in I$ the random variable $X(x)$ is infinitely divisible according to Theorem 3.1 in \cite{Adler1983}; let its characteristics be denoted by $(p_x,A_x,\mu_x)$. The additive sheet is said to be {\em natural} if the mapping $x\mapsto p_x$, which is necessarily continuous, is of bounded variation, or equivalently, if there exists an atomless signed measure $\gamma$ with $p_x=\gamma((0,x])$ for all $x\in I$; here,  we use the convention $(0,x]:=\prod_{i=1}^d I_i$ where, for $x=(x_1,\ldots,x_d)\in I$, $I_i:=(0,x_i]$ when $x_i>0$ and $I_i:=[x_i,0)$ when $x_i<0$. 
The notation of natural additive processes is introduced in Sato \cite{Sato-2004} for the case $d=1$. 

Similarly as for infinitely divisible random measures, we introduce a dynamical aspect in the following definition:
\begin{definition}
	\label{def:LVAS}
	A family $(X(t, \cdot):\, t\ge 0)$ of natural, additive sheets $(X(t,x):\, x\in\R^d)$ is called a {\em L\'evy-valued additive sheet} if for every $x_1,\dots, x_n\in\R^d$ and $n\in\N$, the stochastic process 
	\begin{align*}
		\Big(\big(X(t,x_1),\dots, X(t,x_n)\big):\, t\ge 0\Big)
	\end{align*}
	is a L\'evy process in $\R^n$.
\end{definition}

The wording `L\'evy-valued additive sheet' is motivated by the following result: 
\begin{proposition}
\label{prop:AddSheetOnProduct}
	A L\'evy-valued additive sheet $(X(t,\cdot):t\geq 0)$ forms a natural additive sheet $(X(z): z\in\Rp\times\R^d)$.
\end{proposition}

\begin{proof}
	The domain of definition and Conditions (a), (b) and (d) of Definition \ref{def:AddSheet}  are clearly met. Regarding stochastic continuity, let $(t_n,x_n)_{n\in\N}$ be a sequence in $\Rp\times\R^d$ converging to $(0,x)$. 
For each $n\in\N$ the random variable $X(1,x_n)$ is infinitely divisible, say with characteristics $(p_{x_n},V_{x_n},\mu_{x_n})$. As $X(1,\cdot)$ is a natural, additive sheet, there exists a signed measure $\gamma$ such that $p_{x_n}=\gamma((0,x_n])$. Since the L\'evy process $(X(t,x_n):\, t\ge 0)$ has stationary increments, it follows that each $X(t,x_n)$ has characteristics $(tp_{x_n},tV_{x_n},t\mu_{x_n})$ for every $t\ge 0$.
Theorem 3.1 in \cite{Adler1983} implies that there exist a measure $\Sigma$ on $\Borel(\R^d)$ such that $V_{x_n}=\Sigma((0,x_n])$, and a measure $\nu$ on $\Borel(\R^d\times\R)$ such that, for each $B\in\Borel(\R)$, the mapping $\nu(\cdot\times B)$ is a measure on $\Borel(\R^d)$, and $\mu_{x_n}=\nu((0,x_n]\times\cdot)$. Therefore, the L\'evy symbol of $X(t_n,x_n)$ is given by, for $u\in\R$,
	\begin{align*}
		\Psi_{X(t_n,x_n)}(u)
			= t_n\Big(& iu\gamma((0,x_n])-\tfrac{1}{2}u^2\Sigma((0,x_n]) \\
				& + \int_{(0,x_n]\times\R} \big(e^{iuy}-1-iuy\1_{B_{\R}}(y)\big)\,\nu({\rm d}x,{\rm d}y)
			\Big)
	\,.\end{align*}
As the set $\{x_n:\, n\in\N\}$ is bounded, there exists a bounded box $I\subseteq\R^d$ containing every box $(0,x_n],\,n\in\N$. Thus, we obtain for each $u\in\R$ that
	\begin{align*}
		\abs{\Psi_{X(t_n,x_n)}(u)}
		\leq t_n \left(u\norm{\gamma}_{TV}(I) + \tfrac{1}{2}u^2\Sigma(I) +  
		 \int_{I\times\R} \big(u^2y^2\wedge 1 \big)\,\nu({\rm d}x,{\rm d}y)\right).
	\end{align*}
Finiteness of the right side follows from the fact that the measures are finite on $I$. Therefore, it follows that  $X(t_n,x_n)\to 0$ in probability as $(t_n,x_n)$ converges to $(0,x)$. If $(t_n,x_n)$ is an arbitrary sequence converging to $(t,x)$, stationary increments imply for each $c>0$ that 
\begin{align*}
&	P(\abs{X(t_n,x_n)-X(t,x)}>c)\\
&\qquad \leq P(\abs{X(t_n,x_n)-X(t,x_n)}>\tfrac{c}{2}) + P(\abs{X(t,x_n)-X(t,x)}>\tfrac{c}{2}) \\
&\qquad  = P(\abs{X(t_n-t,x_n)}>\tfrac{c}{2}) + P(\abs{X(t,x_n)-X(t,x)}>\tfrac{c}{2}).
\end{align*}
Consequently, the above established continuity in probability shows the general case.

	The fact that $X(z)$ is natural can be seen from the form of the characteristic function, where we have $p_z=t\gamma((0,x])$ for $z=(t,x)$.
\end{proof}

We are now able to state the link between L\'evy-valued random measures and L\'evy-valued additive sheets. Pedersen showed a similar result in \cite{Pedersen2003}. For the convenience of the reader we present the proof in our setting with some minor modifications. 
\begin{theorem}
\label{thm:NatAddLevyMeasure}
	\begin{enumerate}
		\item[{\rm (a)}] Let $(X(t,\cdot):\, t\ge 0)$ be a L\'evy-valued additive sheet. Then there exists a unique L\'evy-valued random measure $M$ on $\Borel_b(\R^d)$ with atomless control measure $\lambda$ satisfying 	
		\begin{align*}
		M(t,(0,x])=X(t,x)\quad\text{ $P$-a.s.\ for each $t\ge 0$ and $x\in\R^d$.} 
		\end{align*}

		\item[{\rm (b)}] Let $M$ be a  L\'evy-valued random measure on $\Borel_b(\R^d)$ with atomless control measure $\lambda$. Then any lamp and right-continuous modification of the stochastic process $X=(X(t,x):\, t\ge 0,\, x\in\R^d)$ defined by
			\[X(t,x):=
			\begin{cases}
			0, & \text{if $x_j=0$ for some } j=1,\dots, d\\
			M\big(t,(0,x]\big), &\text{else}
			\end{cases}
			\]
		is a L\'evy-valued additive sheet. 
	\end{enumerate}
\end{theorem}

\begin{proof}
	(a) As in the proof of Proposition \ref{prop:AddSheetOnProduct} let $(\gamma, \Sigma, \nu)$ denote the characteristics of $X(1,\cdot)$. 
	By the L\'evy-It\^o decomposition \cite[Theorem~4.6]{Adler1983}, we may write,  	
	\[X(t,x) = t\gamma((0,x]) + G(t,x)
		+ \int_{B_{\R}} y\,\widetilde{N}_{t,x}({\rm d}y)
		+ \int_{B_{\R}^c} y\,N_{t,x}({\rm d}y)
		\qquad P\text{-a.s.},
	\]
where $G$ is a continuous Gaussian additive sheet and, for each $(t,x)\in\Rp\times\R^d$,  $N_{t,x}$ is a Poisson random measure on $\R$ independent of $G$ with intensity measure $t\cdot\nu((0,x]\times\cdot)$, with $\widetilde{N}_{t,x}$ the compensated Poisson random measure. Furthermore, for fixed $B\in\Borel(\R)$ such that $0\notin\widebar{B}$, the process $(N_{t,x}(B):(t,x)\in\Rp\times\R^d)$ is an additive sheet \cite[Proposition~4.4]{Adler1983}.
Let $\mathcal{R}$ be the ring consisting of finite disjoint unions of the form $(s,t]\times(w,x]\times B$ for some $s,t\in\Rp$ with $s< t$, $w,x\in\R^d$ with $w<x$ and $B\in\Borel(\R)$. On $\mathcal{R}$ we define
\[N\left(\bigcup_{i=1}^n(s_i,t_i]\times(w_i,x_i)\times B_i\right):=\sum_{i=1}^n \big(N_{t_i,x_i}(B_i)-N_{s_i,w_i}(B_i)\big),
\]
where the sets $(s_i,t_i]\times(w_i,x_i)\times B_i$, $i=1,\ldots n$ are disjoint.
We see that $N$ is finitely additive, and independently scattered. By the same argument as in Proposition \ref{prop:LVRMisISRM} we obtain the unique extension of $N$ to an independently scattered random measure on $\Rp\times\R^d\times\R$, and it immediately follows that $N$ is a Poisson random measure with intensity $\leb\otimes\nu$.

Let $\widetilde{N}$ be the compensated random measure of $N$, then Proposition \ref{prop:LRM} implies that the mapping $M_0\colon \Borel_b(\Rp\times\R^d)\to L^0(\Omega,P)$ defined by
	\begin{align*}
		M_0(B)
		:= (\leb\otimes\gamma)(B)
		+ \int_{B\times B_{\R}} y\,\widetilde{N}({\rm d}s,{\rm d}x,{\rm d}y)
		+ \int_{B\times B_{\R}^c} y\,N({\rm d}s,{\rm d}x,{\rm d}y),
	\end{align*}
	defines a L\'evy-valued random measure $M_0$ on $\Borel_b(\R^d)$ by $M_0(t,A):=M_0((0,t]\times A)$. It satisfies $M_0(t,(0,x])=X(t,x)-G(t,x)\;P$-a.s. 
	for all $t\ge 0$ and $x\in\R^d$. Uniqueness of $M_0$ follows from that of $N$.
	
Let $\mathcal{R}$ be the ring consisting of finite unions of disjoint half-open boxes of the form $(w,z]$ for some $w,z\in\Rp\times\R^d$ with $w<  z$. On $\mathcal{R}$ we define
\[M_1\left(\bigcup_{i=1}^n(w_i,z_i]\right):=\sum_{i=1}^n \big(G(z_i)-G(w_i)\big),
\qquad w_1< z_1\le w_2 <  \cdots \le w_n<z_n.
\]
Let $(I_n)\subseteq\mathcal{R}$ be a sequence of boxes decreasing to $\emptyset$. Stochastic continuity of $G$ implies 
$G(I_n)\to 0$ in probability as $n\to\infty$. Since $M_1$ is  clearly additive on $\mathcal{R}$,  Theorem 2.15 in \cite{Kallenberg2017} implies,  by considering positive and negative parts separately, that  $M_1$ uniquely extends to an independently scattered random measure on $\Borel_b(\Rp\times\R^d)$.
	
	We now define $M(t,A):=M_0(t,A)+M_1((0,t]\times A)$ for $t\geq 0$ and $A\in\Borel_b(\R^d)$. The random measure $M$ satisfies the statement (a), where the control measure $\lambda$ is atomless by the stochastic continuity of $X$.

	(b) It suffices to check that the process $(X(1,x):x\in\R^d)$ is an additive sheet. Conditions (a) and (b) of Definition \ref{def:AddSheet} are immediately implied by properties of $M$. 
For establishing stochastic continuity, let $(x_n)\subseteq \R^d$ converges to $x$. It follows for fixed $c>0$ that
	\begin{align*}
		P(\abs{X(1,x)-X(1,x_n)}>c)
		= P(\abs{M(1,(x_n,x])}>c)
		\to 0,
	\end{align*}
	by the atomlessness of $\lambda$, as this implies $M(1,\{x\})=0$ $P$-a.s.\ for each $x\in\R^d$.
\end{proof}

\begin{remark}
Theorem \ref{thm:NatAddLevyMeasure} and its proof enables us to conclude a converse implication of Proposition \ref{prop:LRM}. If $M$ is a L\'evy-valued random measure $M$ with atomless control measure $\lambda$, then it satisfies  a L\'evy-It\^o decomposition of the form
	\begin{align}\label{eq.Levy-Ito-for-M}
		M(t,A)
		&= t\gamma(A)
		+ G\big((0,t]\times A\big)\\
		&+ \int_{(0,t]\times A\times B_{\R}} y\,\widetilde{N}({\rm d}s,{\rm d}x,{\rm d}y)
		+ \int_{(0,t]\times A\times B_{\R}^c} y\,N({\rm d}s,{\rm d}x,{\rm d}y),
	\end{align}
where $\gamma$ is a signed measure on $\O$, $G$ is a Gaussian  random measure on $\Borel_b(\Rp\times\O)$ and $N$ is an independent Poisson random measure on $\Borel_b(\Rp\times\O\times\R)$ with compensated part $\widetilde{N}$. 

Furthermore, we see that one does not achieve larger generality by allowing an arbitrary measure space $(U,\mathcal{U},\nu)$ in Proposition  \ref{prop:LRM},  as the Poissonian components can be represented as integrals over $\mathbb{R}$.
\end{remark}

For introducing a stochastic integral of deterministic functions $f\colon \R^d\to \R$  with respect to a L\'evy-valued sheet one could follow the standard approach by starting with simple functions and extending the integral operator by continuity. Instead, for simplifying our presentation, we utilise the correspondence between L\'evy-valued additive sheets and L\'evy valued random measures, established in Theorem \ref{thm:NatAddLevyMeasure}, and refer to the integration for the latter developed in Rajput and Rosinksi \cite{Rajput1989} as presented in  Section \ref{subsec:IDRM}. For a L\'evy-valued additive sheet $(X(t,x):\, t\ge 0, \, x\in \R^d)$ let $M$ denote the corresponding L\'evy-valued random measure on $\Borel_b(\R^d)$ with control measure $\lambda$. Then we define for all 
$f\in L_{M}(\R^d,\lambda)$, $A\in\Borel(\R^d)$ and $t\geq 0$:
\begin{align}\label{eq.def-intX}
	\int_{A} f(x)\,{\rm d}X(t,x)
	:= \int_{A} f(x)\,M(t,{\rm d}x). 
\end{align}

Let $(X(t,x):\, x\in\R^d,\,t\ge 0)$ be a L\'evy-valued additive sheet and $\O\subseteq \R^d$ be open. Then the definition in \eqref{eq.def-intX} allows us to define the same operator $J_{\D}(t)$ as introduced in \eqref{eq.def-J} for a L\'evy-valued additive sheet $X$:
\begin{align}\label{eq.def-J-X}
	J_{\D}(t)\colon  \D(\O)\to L^0(\Omega,P), \quad J_{\D}(t)f=\int_{\O} f(x)\, {\rm d}X(t,x). 
\end{align}
Theorem \ref{th.J-in-D} guarantees that $J_{\D}$ is well-defined and even more, induces a genuine L\'evy process $Y$ in $\D^\ast(\O)$.  Because of the framework of a sheet as a function in $d$ arguments, we can define the operator 
	\begin{align}\label{eq.def-I}
		I_{\D}(t)\colon \D(\O)\to L^0(\Omega,P), \qquad I_{\D}(t)(f)=\int_{\O} f(x)X(t,x)\, {\rm d}x\,. 
	\end{align}
The mapping $I_{\D}$ is well defined because of the lamp property of $X(t,\cdot)$ for each $t\ge 0$ and as each $f\in\D(\O)$ has compact support in $\O$. Lebesgue's dominated convergence theorem shows that $I_{\D}$ is continuous, as every convergent sequence in $\D(\O)$ is uniformly bounded and compactly supported.

The following establishes the relation
	\begin{align*}
		(-1)^{d} I_{\D}(t)(\dot{f})=J_{\D}(t)(f)
		\qquad\text{for all }f\in \D(\O). 
	\end{align*}
In other words, if we neglect the embedding by the operators $I_{\D}$ and $J_{\D}$, we could interpret this result that $M$ is the weak derivative of $X$. This is not very surprising, since, if we adapt notions from classical measure theory, the relation $M(t,(0,x])=X(t,x)$ derived in Theorem \ref{thm:NatAddLevyMeasure}, can be seen that $X$ is the cumulative distribution function of the random measure $M$.

\begin{theorem}\label{th.X-in-D}
For a L\'evy-valued additive sheet $(X(t,\cdot):\,  t\ge 0)$  and an open set $\O\subseteq \R^d$ let $I_{\D}$ be defined by \eqref{eq.def-I}. 
Then there exists a stochastic process $(V(t):\, t\ge 0)$ in $\D^\ast(\O)$ satisfying 
	\begin{align*}
		\scapro{f}{V(t)}=I_{\D}(t)f\qquad\text{for all }f\in\D(\O), \, t\ge 0.
	\end{align*}
	Furthermore, we have the equality 
	\begin{align}
	\label{eq:IDiffJ}
		(-1)^{d} I_{\D}(t)(\dot{f})=J_{\D}(t)(f)
		\qquad\text{for all }f\in \D(\O)\,, 
	\end{align}
	where $\dot{f}=\tfrac{\partial^d }{\partial x_1\cdots \partial x_d}f $ and $J_{\D}(t)$ denotes the operator in \eqref{eq.def-J-X}.
\end{theorem}

\begin{proof} 
We  show that, for each $f\in\D(\O)$, the process $(I_{\D}(t)f:t\geq 0)$ has a c\`adl\`ag modification. First we consider a sequence $(t_n)$ decreasing monotonically to some $t\ge 0$.  Let $K$ be the support of $f$. Then, as $(t_n)$ is bounded, there exists a $C>0$ such that $t_n\in[t,t+C]$ for each $n$. The lamp property of $X$ implies that $X$ is bounded on the compact set $[t,t+C]\times K$. Thus, since 
$X(t_n,x)$ converges to $X(t,x)$ in probability for each $x\in\O$, 
Lebesgue's  dominated convergence theorem (for a stochastically convergent sequence) implies 
	\[\lim_{n\to\infty} \int_{\O} f(x)X(t_n,x)\,{\rm d}x
		= \int_{\O} f(x)X(t,x)\,{\rm d}x. 
	\]
A similar argument establishes that the left limits exists.

The existence of the stochastic process $V$ follows from Theorem 3.2 in \cite{Fonseca-Mora2018} (as $\D(\O)$ is nuclear \cite[Theorem~51.5]{Treves1967} and ultrabornological \cite[Page~447]{Narici2011}).	
	
	To show \eqref{eq:IDiffJ} we use ideas from \cite{Dalang2015}. By the fundamental theorem of calculus, as $f$ has compact support,
		\[f(x) = (-1)^d \int_{\O} \dot{f}(y)\mathbbm{1}_{\{y\geq x\}}\,\mathrm{d}y \qquad\text{for all }x\in \O.
		\]
	By utilising an analogue of Fubini's theorem for L\'evy-valued random measures, as detailed below, we obtain 
		\[\begin{split}
	J_{\D}(t)f = \int_{\O} f(x)\,X(t,\mathrm{d}x)
			& = \int_{\O} \left( (-1)^d \int_{\O} \dot{f}(y)\mathbbm{1}_{\{y\geq x\}}\,\mathrm{d}y \right)\,X(t, \mathrm{d}x) \\
			& = (-1)^d \int_{\O} \left( \int_{\O} \mathbbm{1}_{\{y\geq x\}}\,X(t,\mathrm{d}x) \right)\dot{f}(y)\,\mathrm{d}y \\
			& = (-1)^d \int_{\O} X(t, y)\dot{f}(y)\,\mathrm{d}y\,.
		\end{split}\]
	
We now show the analogue of Fubini's theorem to complete the proof. Let $M$ denote the L\'evy-valued random measure corresponding to $X$ according to Theorem \ref{thm:NatAddLevyMeasure}. The L\'evy-It{\^o} decomposition \eqref{eq.Levy-Ito-for-M} yields that $M$ admits the decomposition 
\begin{align*}
M(t,A)= t\gamma(A) + G(t,A) + M_c(t,A) + M_p(t,A)\quad\text{for all $A\in\Borel_b(\R^d)$ and $t\geq 0$.}
\end{align*}
 Here, $\gamma$ is a signed measure,  $G$ is a pure Gaussian L\'evy-valued random measure with characteristics $(0,\Sigma,0)$, and 
\begin{align*}
M_c(t,A):=\int_0^t \int_{A\times B_{\R}} y\, \widetilde{N}(\mathrm{d}s, \mathrm{d}x, \mathrm{d}y), 
\qquad 
M_p(t,A):=\int_0^t \int_{A\times B_{\R}^c} y\, N(\mathrm{d}s, \mathrm{d}x, 
\mathrm{d}y). 
\end{align*}
The classic Fubini theorem may be applied to $\gamma$. The  L\'evy-valued random measure $M_p$ is a finite random sum and the Fubini result holds trivially. 

For $G$ and the compensated Poisson L\'evy-valued random measure $M_c$ we apply Theorem 2.6 in \cite{Walsh1986}. We note that $(G(t,\cdot)+M_c(t,\cdot):\, t\geq 0)$ forms a martingale-valued measure. Furthermore, $G+M_c$ is orthogonal by the independence of the processes $(G(t,A):t\geq 0)$, $(M_c(t,A):t\geq 0)$, $(M_g(t,B):t\geq 0)$ and $(M_c(t,B):t\geq 0)$ whenever $A,B\in\Borel_b(\R^d)$ are disjoint. The covariance process is given by $Q_t(A,B)=t\big(\Sigma(A\cap B)+\int_{(A\cap B)\times B_{\R}}\abs{y}^2\,\nu({\rm d}x,{\rm d}y)\big)$. 
As the required dominating measure $K$ in  \cite[Theorem~2.6]{Walsh1986}
one can choose $K(A\times B\times(0,t])=t\lambda(A\cap B)$. The required integrability condition follows as $f$ is compactly supported and bounded.
\end{proof}

\begin{remark}
According to Proposition \ref{prop:AddSheetOnProduct}, a L\'evy-valued additive sheet $X$ defines a natural additive sheet $(X(t,x):\, t\ge 0,\, x\in\R^d)$. Due to its lamp trajectories, we can define the mapping
\begin{align*}
	I'_{\D}\colon \D((0,\infty)\times \R^d)\to L^0(\Omega,P), \qquad I'_{\D}(f)=\int_{(0,\infty)\times \R^d} f(t,x)X(t,x)\, {\rm d}t\,{\rm d}x. 
\end{align*}
On the other side, one can conclude as in Theorem 
\ref{thm:NatAddLevyMeasure} or by \cite[Theorem\ 4.1]{Pedersen2003}, that there exists an infinitely divisible  random measure $M'$ on $\Borel_b(\Rp\times\R^d)$ satisfying $M((0,z])=X(z)$ for all 
$z\in\Rp\times\R^d$. Thus, as in Remark \ref{re.M-as-both}, we can define  
	\begin{align*}
		J'_{\D}\colon \D((0,\infty)\times \R^d)\to L^0(\Omega,P),\quad 
		J'_{\D}(g)=\int_{(0,\infty)} \int_{\R^d} g(t,x)\, M'({\rm d}t,{\rm d}x). 
	\end{align*}
One can conclude as in the proof of Theorem \ref{th.X-in-D} that there 
exists a genuine random variable $W$ in $\D^\ast((0,\infty)\times \R^d)$ satisfying 
\begin{align*}
		\scapro{g}{W}=I'_{\D}(g)\qquad\text{for all }g\in\D((0,\infty)\times\R^d).
\end{align*}
	Furthermore, we have the equality 
	\begin{align*}
		(-1)^{d} I'_{\D}(\dot{g})=J'_{\D}(g)
		\qquad\text{for all }g\in\D((0,\infty)\times\R^d). 
	\end{align*}
\end{remark}

\section{Cylindrical L\'evy processes}
\label{sec:CLPandLRM}

The concept of cylindrical L\'evy processes is introduced in \cite{Applebaum2010}. It naturally generalises the notation of cylindrical Brownian motion, based on the theory of cylindrical measures and cylindrical random variables.  Here, a cylindrical random variable $Z$ on a Banach space $F$ is a linear and continuous mapping $ Z\colon F^{\ast} \rightarrow L^0(\Omega,P)$, where $F^\ast$ denotes the dual space of $F$. The characteristic function of $Z$ is defined by $\phi_Z (f)=E[\exp(iZf)]$ for all $f\in F^\ast$. In many cases, we will choose $F=L^p(\O,\zeta)$ for some $p\ge 1$ and an arbitrary locally finite Borel measure $\zeta$. In this case $F^\ast=L^{p^\prime}(\O,\zeta)$ for $p^\prime:=\tfrac{p}{p-1}$.

\begin{definition}
A family $(L(t):\, t\ge 0)$ of cylindrical random variables $L(t)\colon F^\ast\to L^0(\Omega,P)$ is called a \emph{cylindrical L\'evy process}  if for all $f_1, ... , f_n \in F^\ast$ and $n\in \mathbb{N}$, the stochastic process $((L(t)f_1, ... , L(t)f_n):\, t \ge 0)$ is a L\'evy process in $\mathbb{R}^n$. 
\end{definition}
The characteristic function of a cylindrical L\'evy process $(L(t):\, t\ge 0)$ is given by
\[\phi_{L(t)} \colon F^\ast \rightarrow \mathbb{C}, \qquad \phi_{L(t)}(f)=\exp\big(t\Psi_L(f)\big),\]
for all $t\ge 0$. Here,  $\Psi_L \colon F^\ast \rightarrow \mathbb{C}$ is called the (cylindrical) symbol of $L$, and is of the form
\begin{align*}
	\Psi_L(f) = ia(f) - \tfrac{1}{2}\langle f, Qf\rangle +\int_F\left(e^{i\langle g, f \rangle}-1-i\langle g, f \rangle \1_{B_{\mathbb{R}}}(\langle g, f \rangle)\right)\mu ({\rm d} g),
\end{align*}
where $a \colon F^\ast \rightarrow \mathbb{R}$ is a continuous mapping with $a(0)=0$,
the mapping $Q \colon F^\ast \rightarrow F^{\ast\ast}$ is a positive, symmetric operator  and $\mu$ is a finitely additive  measure on $\mathcal{Z}(F)$ satisfying
\[\int_F \left( \langle g, f \rangle^2 \wedge 1 \right) \mu({\rm d} g) < \infty \qquad \mathrm{for\;all\;}f \in F^\ast.\]
Here, ${\mathcal Z}(F)$ is the algebra of all sets of the form $\{g\in F:\,(\scapro{g}{f_1},\dots, \scapro{g}{f_n})\in B\}$ for some $f_1,\dots, f_n\in F^\ast$, $B\in\Borel(\R^n\setminus\{0\})$ and $n\in\N$.
We call $(a, Q, \mu)$ the \emph{(cylindrical) characteristics of $L$}. 

\begin{theorem}
\label{th.L-by-M}
	Let $M$ be a L\'evy-valued random measure on $\Borel_b(\O)$ with characteristics $(\gamma,\Sigma,\nu)$ and control measure $\lambda$. If $F$ is a Banach space for which $F^\ast$ is continuously embedded into $L_{M}(\O,\lambda)$, and the simple functions are dense in $F^\ast$, then 
	\begin{align}\label{eq.def-L-by-M}
		L(t)f:=\int_\O f(x)\, M(t,{\rm d}x) \qquad \text{for all } f\in F^\ast, 
	\end{align}
	defines a cylindrical L\'evy processes $L$ in  $F$. In this case, the characteristics $(a,Q,\mu)$ of $L$ is given by
	\begin{align*}
		a(f) &= \int_{\O}f(x)\,\gamma(\mathrm{d}x)
		+ \int_{\O\times\mathbb{R}} f(x)y \big(\mathbbm{1}_{B_{\mathbb{R}}}(f(x)y)-\mathbbm{1}_{B_{\mathbb{R}}}(y) \big) \,\nu(\mathrm{d}x,\mathrm{d}y),\\
		\scapro{Qf}{f} & = \int_{\O}(f(x))^2\,\Sigma(\mathrm{d}x),
\qquad	\qquad	\mu\circ\scapro{f}{\cdot}^{-1} = \nu\circ\chi_{f}^{-1},
	\end{align*}
	for each $f\in F^\ast$,	where $\chi_{f}\colon \O\times\R\to \R$ is defined by  $\chi_{f}(x,y):=f(x)y$.
\end{theorem}
\begin{proof}
	Lemma \ref{lem:JnDimLevy} shows that $L$ is a cylindrical L\'evy process in $F$. 
The claimed characteristics follows from \eqref{eq:MIntChars} after rearranging the 
terms accordingly. 	
\end{proof}

The integration theory developed in \cite{Rajput1989} and presented at the end of Section 3 guarantees that \eqref{eq.def-L-by-M} 
is well defined for every $f\in L_M$. However, in order to be in the framework of cylindrical L\'evy processes we need that the domain is the dual of a Banach space (or alternatively a nuclear space). Since the Musielak-Orlicz space $L_M$ is not in general the dual of a Banach space, for the hypothesis of Theorem \ref{th.L-by-M} we require the existence of the Banach space $F$ with $F^\ast$ continuously embedded in $L_M$. If the control measure $\lambda$ of $M$ is finite on $\O$, then  Lemma \ref{lem:L2Embedding} gives us that $L^2(\O,\lambda)$ is continuously embedded in $L_M(\O,\lambda)$. It is possible as illustrated in the following example to relax the condition on finiteness of $\lambda$, but also the same example shows that there are cases where the finiteness of $\lambda$ is necessary for any $L^p$ space to be continuously embedded.

\begin{example}
	We return again to Example \ref{ex.Balan}; let $M$ be the $\alpha$-stable random measure for some $\alpha\in(0,2)$,  where now we consider the domain of definition to be $\Borel_b(\O)$ for a general $\O\in\Borel(\R^d)$. We consider the symmetric case $p=q=\tfrac12$, where the characteristics of $M$ is given by $(0,0,\leb\otimes\nu_{\alpha})$ and the control measure by $\lambda(A)=\tfrac{2}{2-\alpha}\leb(A)$, $A\in\Borel(\O)$. One calculates from \eqref{eq.def-Phi-M} that $L_M(\O,\lambda)= L^\alpha(\O,\leb)$; see \cite[Lemma~4]{Balan2014}. 

Thus, if $\alpha\in (1,2)$ then we can always choose $F=L^{\alpha^\prime}(\O,\leb)$. 
	If $\alpha\in (0,1]$ and $\O$ is bounded we can choose $F=L^p(\O,\leb)$
	for any $p >1$ since $\abs{f(x)}^\alpha\le 1+\abs{f(x)}^{p'}$. However, if $\leb(\O)=\infty$ and $\alpha\leq 1$ then no $L^p$ space is embedded in $L_M$ for $p>1$.
	
Assume $\alpha\in (1,2)$. Then Theorem \ref{th.L-by-M} implies that \eqref{eq.def-L-by-M} defines a cylindrical L\'evy process $L$ in $F=L^{\alpha^\prime}(\O,\leb)$, and its symbol is given by
\begin{align*}
	\Psi_L(f)
	&=i\int_{\O\times\R}f(x)y\big(\1_{B_{\R}}(f(x)y)-\1_{B_{\R}}(y) \big) \,{\rm d}x\nu_{\alpha}({\rm d}y)\\
	&\qquad\qquad + \int_{\R} \left(e^{iu}-1-iu\1_{B_{\R}}(u)\right)\,(\mu\circ\scapro{f}{\cdot}^{-1})({\rm d}u)\\
	&=\int_{\O\times\R}\left(e^{if(x)y}-1-if(x)y\mathbbm{1}_{B_{\mathbb{R}}}(y)\right)\,
		{\rm d}x\nu_{\alpha}({\rm d}y)
= -C_\alpha \norm{f}_{L^\alpha(\O,\leb)}^\alpha, 
\end{align*}
where $C_{\alpha}=\frac{\Gamma(2-\alpha)}{1-\alpha}\cos\frac{\pi\alpha}{2}$ if $\alpha\neq 1$ and $C_\alpha=\frac{\pi}{2}$ if $\alpha=1$.
\end{example}


We now turn to the question of which cylindrical L\'evy processes induce L\'evy-valued random measures. 
For this purpose we introduce the following:
\begin{definition}
\label{def:IndScatt}
A cylindrical L\'evy process $(L(t):\, t\ge 0)$ in $L^p(\O,\zeta)$ for some $p\ge 1$ is called \textup{\textbf{independently scattered}} if for any disjoint sets $A_1, \ldots ,A_n \in \Borel_b(\mathcal{O})$ and $n\in\mathbb{N}$, the random variables $L(t)\mathbbm{1}_{A_1},\ldots ,L(t)\mathbbm{1}_{A_n}$ are independent for each $t\geq 0$.
\end{definition}

\begin{theorem}
\label{prop:CLPtoIDRM}
An independently scattered cylindrical L\'evy process $(L(t):t\geq 0)$ in $L^p(\mathcal{O},\zeta)$ for some $p\ge 1$ defines by 
\begin{align}\label{eq.L-defines-M}
M(t,A):= L(t)\mathbbm{1}_A
\qquad\text{for all }t\geq 0, A\in \Borel_b(\O), 
\end{align}
a L\'evy-valued random measure $M$ on $\Borel_b(\mathcal{O})$.
\end{theorem}

\begin{proof}
	For each $t\geq 0$, the map $M(t,\cdot)\colon \Borel_b(\mathcal{O})\to L^0(\Omega,P)$ is well-defined and $M(t,A)$ is an infinitely divisible random variable for each $A\in\Borel_b(\O)$. Let $(A_k)_{k\in\mathbb{N}}$ be a sequence of disjoint sets in $\Borel_b(\mathcal{O})$ such that $A:=\bigcup_{k\in\N}A_k\in\Borel_b(\mathcal{O})$. Then, for each $t\geq 0$, by the linearity and continuity of $L(t)$ we have
	\[M(t,A) 
		= \lim_{n\to\infty} L(t)\mathbbm{1}_{\bigcup_{k=1}^n A_k} 
		= \lim_{n\to\infty} \sum_{k=1}^n L(t)\mathbbm{1}_{A_k}
		= \lim_{n\to\infty} \sum_{k=1}^n M(t,A_k)\,,
	\]
	with the limit in probability and thus almost surely by independence. Clearly,  $M(t,\cdot)$ is independently scattered for each $t\ge 0$, and $\big(M(\cdot, A_1),\ldots,M(\cdot,A_n)\big)$ is a L\'evy process for each $A_1,\ldots A_n\in\Borel_b(\O)$.
\end{proof}

\begin{theorem}
\label{thm:IndScatChars}
	Let $(L(t):\, t\geq 0)$ be a cylindrical L\'evy process in $L^p(\mathcal{O},\zeta)$ for some $p\ge 1$. Then $L$ is independently scattered if and only if its  symbol is of the form 
	\begin{equation}
	\label{eq:IndScatChars}
		\begin{split}
			\Psi_L(f)
			& = i \int_{\mathcal{O}} f(x) \,\gamma(\mathrm{d}x)
			- \tfrac{1}{2} \int_{\mathcal{O}} f^2(x) \,\Sigma(\mathrm{d}x) \\
			& + \int_{\mathcal{O}\times\mathbb{R}} \left(e^{if(x)y}-1-if(x)y\mathbbm{1}_{B_{\mathbb{R}}}(y)\right) \,\nu(\mathrm{d}x,\mathrm{d}y)
			\,,\quad f\in L^{p'}(\mathcal{O},\zeta),
		\end{split}
	\end{equation}	
for a signed measure $\gamma$ on $\Borel_b(\mathcal{O})$, a measure $\Sigma$ on $\Borel_b(\mathcal{O})$ and a $\sigma$-finite measure $\nu$ on $\Borel(\mathcal{O}\times\mathbb{R})$ such that for each $B\in\Borel_b(\mathcal{O})$, $\nu(B\times\cdot)$ is a L\'evy measure on $\mathbb{R}$.
\end{theorem}

\begin{proof} 
	If $L$ is independently scattered then 
	Theorem \ref{prop:CLPtoIDRM} implies that $L$ defines a L\'evy-valued random measure $M$ by \eqref{eq.L-defines-M}. Denote the characteristics of $M$ by $(\gamma,\Sigma,\nu)$ and its control measure by $\lambda$. For a simple function $f$ of the form \eqref{eq.def-simple} we obtain
	\begin{align}\label{eq.L-and-M}
		L(t)(\1_A f) &= \sum_{i=1}^n \alpha_i L(t)(\1_{A_i}\1_A )
		= \sum_{i=1}^n \alpha_i M(t,A_i\cap A) 
		= \int_{A} f(x) \,M(t,{\rm d}x)\,.
	\end{align}
For an arbitrary function $f\in L^{p'}(\O,\zeta)$ let $(f_n)_{n\in\N}$ be a sequence of simple functions 
converging to $f$ both pointwise $\zeta$-almost everywhere and in  $L^{p'}(\O,\zeta)$. We note that, as $L(t)\1_A=0$ whenever $\zeta(A)=0$, $\zeta$-null sets have null $\lambda$-measure, and thus we have $f_n\to f$ pointwise $\lambda$-almost everywhere. Since $L(t)\1_Af_n \to L(t)\1_Af$ in probability for each $A\in \Borel(\O)$, it follows from \eqref{eq.L-and-M} that $f\in L_M(\O,\lambda)$ and $L(t)f=\int_{\O}f(x)\,M(t,{\rm d}x)$. We obtain the stated form of the characteristic function of $L$ by \eqref{eq:MIntChars}.
	
	Conversely, if the L\'evy symbol is given by \eqref{eq:IndScatChars}, then this form implies for any disjoint sets $A_1,\ldots , A_n \in\Borel_b(\O)$ that 
	\begin{align*}
		\Psi_L\left(\sum_{k=1}^n u_k \1_{A_k}\right)
		=\sum_{k=1}^n \Psi_L(u_k\1_{A_k})
		\qquad\text{for all }u_1,\dots, u_n\in\R.
	\end{align*}
	Consequently, we obtain for the characteristic function of the random vector 
	$X:=(L(1)\1_{A_1},\dots ,L(1)\1_{A_n})$ for all $u=(u_1,\ldots u_n)\in\mathbb{R}^n$, that
	\begin{align*}
		\phi_X(u)=\phi_{L(1)}(u_1\mathbbm{1}_{A_1}+\cdots +u_n\mathbbm{1}_{A_n})
		&=e^{i\Psi_L(u_1\mathbbm{1}_{A_1}+\cdots +u_n\mathbbm{1}_{A_n})}\\
		&=\phi_{L(1)\mathbbm{1}_{A_1}}(u_1) \cdots \phi_{L(1)\mathbbm{1}_{A_n}}(u_n),
	\end{align*}
which shows that $L$ is independently scattered.
\end{proof}

Applying Theorem \ref{prop:CLPtoIDRM} to a given cylindrical L\'evy process $L$ on $L^p(\O,\zeta)$ gives the corresponding L\'evy valued random measure $M$, say with control measure $\lambda$. The first part of the proof of Theorem \ref{thm:IndScatChars} shows that $L^p(\O,\zeta)$ is a subspace of $L_M(\O,\lambda)$. The following result guarantees that the embedding is continuous in non-degenerated cases. 
\begin{proposition}
	Let $L$ be an independently scattered cylindrical L\'evy process in $L^p(\O,\zeta)$ with symbol of the form \eqref{eq:IndScatChars} and $M$ the corresponding L\'evy-valued random measure with control measure $\lambda$. If the measures $\gamma,\Sigma$ and $\nu$ are such that for each $A\in\Borel_b(\O)$ with $\Sigma(A)=0$ and $\nu(A\times B)=0$ for each $B\in\Borel(\R)$ bounded away from $0$, we have $\norm{\gamma}_{\rm TV}(A)=0$, then  $L^{p'}(\O,\zeta)$ is continuously embedded into $L_M(\O,\lambda)$. 
\end{proposition}

\begin{proof}
	By the first part of the proof of Theorem \ref{thm:IndScatChars} we have $L^{p'}(\O,\zeta)\subseteq L_M(\O,\lambda)$, and, furthermore, the canonical injection $\iota\colon L^{p'}(\O,\zeta)\to L_M(\O,\lambda)$ is well defined, as the $\zeta$-equivalence class of $f$ is a subset of the $\lambda$-equivalence class of $f$. For each $t\geq 0$ we consider the operator $J(t)\colon L_M(\O,\lambda)\to L^0(\Omega,P)$ defined in \eqref{eq:JLMtoL0} and we see that $L(t)$ satisfies the factorisation $L(t) = J(t)\circ\iota$. 

For establishing $\ker(J(t))=\{0\}$,  let $f\in L_M(\O,\lambda)$ satisfy $J(t)f=0$. Then, by considering only the real part of the characteristic function of $J(t)f$, we have for every $u\in\R$
	\begin{align}
	\label{eq:JtImaginary}
		- \tfrac{1}{2}u^2 \int_{\mathcal{O}} f^2(x) \,\Sigma(\mathrm{d}x) 
		+ \int_{\mathcal{O}\times\mathbb{R}} \big(\cos(uf(x)y)-1\big) \,\nu(\mathrm{d}x,\mathrm{d}y)
		= 0. 
	\end{align}
As both terms are non-positive, we obtain that $f=0$ $\Sigma$-a.e.\ and the function $z(x,y):=f(x)y$ satisfies $z=0$ $\nu$-a.e. In particular, for the set $A:=\{x\in\O:f(x)\neq 0\}$ we have $\Sigma(A)=0$ and $\nu(A\times B)=0$ for any $B\in\Borel(\R)$ bounded away from $0$. The hypothesis on $\gamma$ thus leads to $\lambda(A)=0$, which shows $\ker(J(t))=\{0\}$. 

Let $(f_n)$ be a sequence in $L^{p'}(\O,\zeta)$ converging to $f\in L^{p'}(\O,\zeta)$ and assume 
that $\iota f_n$ converges to some $g\in L_M(\O,\lambda)$.   As $\lim_{n\to\infty} J(t)(\iota f_n) = J(t)g$ and $\lim_{n\to\infty} L(t)f_n = L(t)f = J(t)(\iota f)$, we derive $J(t)(g-\iota f)=0$. 
Since $J(t)$ is injective, we conclude $g=\iota f$ $\lambda$-a.e., and the closed graph theorem implies the continuity of $\iota$.
\end{proof}

\begin{example}
Peszat and Zabczyk in \cite[Section~7.2]{Peszat2007} define the impulsive cylindrical process in $L^2(\mathcal{O},\Borel(\mathcal{O}),\zeta)$ by 
	\[L(t)f:=\int_0^t\int_{\O}\int_{\R} f(x)y\,\widetilde{N}({\rm d}s,{\rm d}x,{\rm d}y),
	\]
	where $N$ is a Poisson random measure on $\Rp\times\O\times\R$ with intensity $\leb\otimes\zeta\otimes\mu$ for a L\'evy measure $\mu$ on $\Borel(\R)$; see also \cite[Ex.\ 3.6]{Applebaum2010}. Since its symbol is given by 
 	\[\Psi_L(f) =
		\int_{\mathcal{O}}\int_{\mathbb{R}} [e^{if(x)y}-1-if(x)y] \,\mu(\mathrm{d}y)\zeta(\mathrm{d}x),
	\]
Theorem \ref{thm:IndScatChars} guarantees that $L$ is independently scattered. 
\end{example}

Finally, we note that the class of independently scattered cylindrical L\'evy processes is a strict subclass, as the following counter-example shows:
\begin{example}
Let $(\ell_k)_{k\in\N}$ be a sequence of independent, identically distributed, real-valued L\'evy processes. Assume for simplicity that $\ell_k$ is symmetric with characteristics $(0,0,\bar{\nu})$, with $\bar{\nu}\neq 0$, and satisfies $E[\abs{\ell_k(1)}^2]<\infty$. Let $(e_k)_{k\in\N}$ be an orthonormal basis of $L^2((0,1),\leb)$ such that $e_1\equiv 1$ (such bases include the standard polynomial and trigonometric bases). 
It follows from Lemma 4.2 in \cite{Riedle2015} that 
\[L(t)f:= \sum_{k=1}^{\infty} \scapro{f}{e_k} \ell_k(t),
\qquad \text{for all }f\in L^2((0,1),\leb), 
	\]
defines a cylindrical L\'evy process $L$, say with characteristics $(0,0,\mu)$. 

Assume for a contradiction that $L$ is independently scattered and fix two disjoint sets $A,B\in\Borel((0,1))$ with $\leb(A)>0$ and $\leb(B)>0$. Thus, $\scapro{\1_A}{e_1}=\leb(A)>0$ and $\scapro{\1_B}{e_1}=\leb(B)>0$. 

The L\'evy measure of the 
L\'evy process $\big((L(t)\1_A,L(t)\1_B):\, t\ge 0\big)$  in $\R^2$ is given by 
$\mu\circ \pi_{\1_A,\1_B}^{-1}$. As $L(1)\1_A$ and $L(1)\1_B$ are independent, it follows from the uniqueness of the characteristic functions that 
\begin{align*}
\mu\circ \pi_{\1_A,\1_B}^{-1}
=((\mu\circ\pi_{\1_A}^{-1})\otimes \delta_0)
+(\delta_0\otimes \mu\circ\pi_{\1_B}^{-1})),
\end{align*}
where $\mu\circ\pi_{\1_A}^{-1}$ is the L\'evy measure of $(L(t)\1_A:\, t\ge 0)$ and 
$\mu\circ\pi_{\1_B}^{-1}$ is the L\'evy measure of $(L(t)\1_B:\, t\ge 0)$.
It follows in particular that 
\begin{align}\label{eq.2-leb=0}
\mu\circ \pi_{\1_A,\1_B}^{-1}\big( \R\setminus \{0\} \times  \R\setminus\{0\}\big)=0. 
\end{align}
On the other hand, Lemma 4.2 in \cite{Riedle2015} implies that 
\begin{align*}
\mu\circ \pi_{\1_A,\1_B}^{-1}
= \sum_{k=1}^{\infty}(\bar{\nu}\circ r_k^{-1}), 
\end{align*}
where $r_k\colon \R\to\R^2$ is defined by $r_k(x)=\big(\scapro{\1_A}{e_k}x,\, 
\scapro{\1_B}{e_k}x\big)$.  
It follows from \eqref{eq.2-leb=0} that 
\begin{align*}
0=\big(\bar{\nu}\circ r_1^{-1}\big)\big( \R\setminus \{0\} \times  \R\setminus\{0\}\big)=\bar{\nu}(\R\setminus\{0\})>0\, ,
\end{align*}
which results in a contradiction. 
\end{example}



\end{document}